\UseRawInputEncoding
\documentclass[11pt]{amsart}
\usepackage[T1]{fontenc}
\usepackage{amsmath}
\usepackage{amssymb}
\usepackage{amsthm}
\usepackage{enumitem}
\usepackage{tikz}
\usetikzlibrary{patterns}
\usepackage{setspace}

\if@dg@section
\def\sectionautorefname~{\S}

\fi

\newtheorem{theorem}{Theorem}[section]
\newtheorem{definitio}[theorem]{Definition}
\newenvironment{definition}{\begin{definitio} \rm }{\end{definitio}}
\newtheorem{rem}[theorem]{Remark}

\newtheorem{ex}[theorem]{Example}
\newenvironment{example}{\begin{ex} \rm }{\end{ex}}
\newtheorem{lemma}[theorem]{Lemma}
\newtheorem{proposition}[theorem]{Proposition}
\newtheorem{corollary}[theorem]{Corollary}

\newcommand{\fl}[1]{\textrm{\raisebox{-0.14cm}{\begin{tikzpicture}\draw (0.5,0) node{$\longrightarrow$}; \draw (0.5,0.3) node{$#1$};
  \end{tikzpicture}}}}

\newcommand{\F}{\mathbb{F}}

\newcommand{\Z}{\mathbb{Z}}
\newcommand{\Zt}{\Z[t^{\pm1}]}
\newcommand{\Q}{\mathbb{Q}}
\newcommand{\C}{\mathbb{C}}
\newcommand{\K}{\mathbb{K}}
\newcommand{\HH}{\mathcal{H}}
\newcommand{\Seq}{\mathcal{S}}

\newcommand{\Int}[1]{\mathrm{Int}{(#1)}}

\newcommand{\p}[1]{\widehat{#1}}
\newcommand{\pl}{{\scriptscriptstyle{\Lambda}}}
\newcommand{\pf}{{\scriptscriptstyle{\F}}}
\newcommand{\pr}{{\scriptscriptstyle{R}}}
\newcommand{\hS}{h_{\scriptscriptstyle{\Sigma}}}
\newcommand{\hYS}{h_{\scriptscriptstyle{Y,\Sigma}}}
\newcommand{\hXY}{h_{\scriptscriptstyle{X,Y}}}
\newcommand{\hY}{h_{\scriptscriptstyle{Y}}}
\newcommand{\hX}{h_{\scriptscriptstyle{X}}}

\newcommand{\ha}{h_\alpha}
\newcommand{\hb}{h_\beta}
\newcommand{\hc}{h_\gamma}
\newcommand{\hab}{h_{\alpha\beta}}
\newcommand{\hbc}{h_{\beta\gamma}}
\newcommand{\hca}{h_{\gamma\alpha}}

\definecolor{vert}{RGB}{0,205,0}

\title{Torsions and intersection forms of $4$--manifolds from trisection diagrams}

\author{Vincent Florens and Delphine Moussard}

\date{}

\begin{document}

\begin{abstract}
Gay and Kirby introduced trisections which describe any  closed oriented smooth 4--manifold $X$ as a union of three four-dimensional handlebodies. A trisection is encoded in a diagram, namely three collections of curves in a closed oriented surface $\Sigma$, guiding the gluing of the handlebodies. Any morphism $\varphi$ from $\pi_1(X)$ to a finitely generated free abelian group induces a morphism on $\pi_1(\Sigma)$. We express the twisted homology and Reidemeister torsion of $(X;\varphi)$ in terms of the first homology of $(\Sigma;\varphi)$ and the three subspaces generated by the collections of curves. We also express the intersection form of $(X;\varphi)$ in terms of the intersection form of $(\Sigma;\varphi)$.
\vspace{1ex}

\noindent \textbf{MSC 2010}: 57Q10 57M99 
\end{abstract}

\maketitle

\section{Introduction}

In \cite{GK}, Gay and Kirby proved that any smooth closed oriented $4$--manifold can be trisected into three $4$--dimensional handlebodies with $3$--dimensional handlebodies as pairwise intersections and a closed surface as triple intersection. Such a decomposition can be encoded in a trisection diagram given by three families of curves on this surface. 
This could be thought as a $4$--dimensional analogue of Heegaard splittings and diagrams and allows to use classical $2$ and $3$--dimensional technics to describe invariants of $4$--manifolds. 

Feller, Klug, Schirmer and Zemke \cite{FKSZ18} recently expressed the homology and the intersection form of a closed $4$--manifold $X$ in terms of a trisection diagram. In this paper, we extend their results to the case of coefficients twisted by a group homomorphism $\varphi:\pi_1(X) \rightarrow G$, where $G$ is a finitely generated free abelian group, and we express the Reidemeister torsion of $(X;\varphi)$ in terms of the diagram. More precisely,  we introduce a short finite dimensional complex, whose spaces are given by the first twisted homology module of the surface $\Sigma$ of the diagram and its subspaces generated by the curves of the diagram. We show that the homology of $(X,\varphi)$ and the related Reidemeister torsion are those of this complex. Using the associated expression of the homology modules of $(X;\varphi)$, we express the intersection form of $(X,\varphi)$ in terms of the intersection form of $(\Sigma,\varphi)$. We also give a method to compute the Alexander polynomial of $(X;\varphi)$ from a trisection diagram.

Our approach is different from \cite{FKSZ18}: while they use a handle decomposition of the manifold associated with the trisection, we work directly with the trisection itself ---a similar method was developed by Ranicki in \cite{Ra09} to compute the signature of a $4$--manifold given as a triple union. We also review the non-twisted homology and intersection form (corresponding to a trivial morphism $\varphi$) from this point of view; this yields especially a simple expression of $H_1(X;\Z)$ and explicit representatives of the homology classes.

\subsection*{Plan of the paper.}
In Section~\ref{secstatements}, we state the main results of the paper.  
In Section~\ref{secprelim}, we recall some definitions and facts related to the twisted homology and the process of reconstruction of the 4--manifold $X$ from the trisection diagram; we also fix some notations. In Section~\ref{sechomologyX}, we compute the homology of $X$ with coefficients in $\Z$. In Section~\ref{sechomologyXtwist}, we describe the twisted homology of $(X;\varphi)$ for a non-trivial $\varphi$. The torsion is treated in Sections~\ref{secTorsion} and~\ref{secTorsionXstar}. Section~\ref{secintform} is devoted to intersection forms. Finally, in Section~\ref{secex}, we illustrate the results with explicit examples.

\subsection*{Acknowledgements.}
The first author was partially supported by the ANR Project LISA 17-CE40-0023-01.
While working on the contents of this paper, the second author has been supported by a Postdoctoral Fellowship of the {\em Japan Society for the Promotion of Science}. She is grateful to Tomotada Ohtsuki and the {\em Research Institute for Mathematical Sciences} for their support. She is now supported by the {\em R\'egion Bourgogne Franche-Comt\'e} project ITIQ--3D. She thanks Gw\'ena\"el Massuyeau and the {\em Institut de Math\'ematiques de Bourgogne} for their support.

\subsection*{Conventions.} 
The boundary of an oriented manifold with boundary is oriented with the ``outward normal first'' convention. 
We also use this convention to define the co-orientation of an oriented manifold embedded in another oriented manifold. \\ 
We use the same notation for a curve $\nu$ in a manifold, its homotopy class and its homology class, precising the one we consider if it is not clear from the context.\\
If $U$ and $V$ are transverse integral chains in a manifold $M$ such that $\dim(U)+\dim(V)=\dim(M)$, define the sign $\sigma_x$ of an intersection point $x\in U\cap V$ in the following way. Construct a basis of the tangent space $T_xM$ of $M$ at $x$ by taking an oriented basis of the normal space $N_xU$ followed by an oriented basis of $N_xV$. Set $\sigma_x=1$ if this basis is an oriented basis of $T_xM$ and $\sigma_x=-1$ otherwise. Now the algebraic intersection number of $U$ and $V$ in $M$ is $\langle U,V\rangle_M=\sum_{x\in U\cap V}\sigma_x$.

\section{Statement of the results} \label{secstatements}

Let $X$ be a closed, connected, oriented, smooth 4--manifold. A $(g,k)$--\emph{trisection of} $X$ is a decomposition $X=X_1\cup X_2\cup X_3$ such that 
\begin{itemize}
 \item $X_i\simeq\natural^k(S^1\times B^3)$ is a $4$--dimensional handlebody for each $i$,
 \item $X_i \cap X_j \simeq\natural^g(S^1\times D^2)$ is a $3$--dimensional handlebody for all $i\neq j$,
 \item $\Sigma=X_1 \cap X_2 \cap X_3$ is a closed surface of genus $g$.
\end{itemize}

Note that $\partial X_i \simeq \sharp^k (S^1 \times S^2)$ and $(\Sigma,X_i \cap X_{i-1},X_i \cap X_{i+1})$ is a genus $g$ Heegaard splitting of $\partial X_i$, where indices are understood modulo 3.
A \emph{trisection diagram} consists of three systems $(\alpha_i)_{1\leq i\leq g}$, $(\beta_i)_{1\leq i\leq g}$ and $(\gamma_i)_{1\leq i\leq g}$ 
of disjoint simple closed curves on the standard closed genus $g$ surface $\Sigma$ such that each one is a complete system of meridians of a handlebody 
of the trisection, respectively $H_\alpha:=X_3 \cap X_1$, $H_\beta:=X_1 \cap X_2$ and $H_\gamma:=X_2 \cap X_3$. 
\begin{figure}[htb] \label{fig}
\begin{center}
\begin{tikzpicture} [scale=0.5]
\draw (0,0) ..controls +(0,1) and +(-2,1) .. (4,2) ..controls +(2,-1) and +(-2,-1) .. (8,2) ..controls +(2,1) and +(0,1) .. (12,0);
\draw (0,0) ..controls +(0,-1) and +(-2,-1) .. (4,-2) ..controls +(2,1) and +(-2,1) .. (8,-2) ..controls +(2,-1) and +(0,-1) .. (12,0);

\draw (2,0) ..controls +(0.5,-0.25) and +(-0.5,-0.25) .. (4,0);
\draw (2.3,-0.1) ..controls +(0.6,0.2) and +(-0.6,0.2) .. (3.7,-0.1);
\draw[color=red] (3,-0.2) ..controls +(0.2,-0.5) and +(0.2,0.5) .. (3,-2.3);
\draw[dashed,color=red] (3,-0.2) ..controls +(-0.2,-0.5) and +(-0.2,0.5) .. (3,-2.3);
\draw[color=blue] (3,0)ellipse(1.6 and 0.8);
\draw[color=green] (1,0) .. controls +(0,2) and +(-1,0) .. (6,0.8) .. controls +(0.5,0) and +(-1,-1) .. (8,2);
\draw[color=green] (1,0) .. controls +(0,-2) and +(-1,-0.5) .. (7,0) .. controls +(0.5,0.3) and +(-0.5,0.5) .. (8.5,0);
\draw[color=green,dashed] (8,2) .. controls +(0.5,-0.5) and +(0,1) .. (8.5,0);

\draw (8,0) ..controls +(0.5,-0.25) and +(-0.5,-0.25) .. (10,0);
\draw (8.3,-0.1) ..controls +(0.6,0.2) and +(-0.6,0.2) .. (9.7,-0.1);
\draw[color=red] (9,-0.2) ..controls +(0.2,-0.5) and +(0.2,0.5) .. (9,-2.3);
\draw[dashed,color=red] (9,-0.2) ..controls +(-0.2,-0.5) and +(-0.2,0.5) .. (9,-2.3);
\draw[color=blue] (9,0)ellipse(1.6 and 0.8);
\draw[color=green] (11,0) .. controls +(0,-2) and +(1,0) .. (6,-0.8) .. controls +(-0.5,0) and +(1,1) .. (4,-2);
\draw[color=green] (11,0) .. controls +(0,2) and +(1,0.5) .. (5,0) .. controls +(-0.5,-0.3) and +(0.5,-0.5) .. (3.5,0);
\draw[color=green,dashed] (4,-2) .. controls +(-0.5,0.5) and +(0,-1) .. (3.5,0);
\end{tikzpicture}
\caption{A trisection diagram for $S^2\times S^2$} \label{figtrisection diagram}
\end{center}
\end{figure}

\subsection{Homology} \label{sechomol}

We compute the homology of $X$ in terms of the trisection diagram. We separate the case of non-twisted coefficients (corresponding to a trivial $\varphi$) and the twisted case. 

The first result is close to \cite[Theorem 3.1]{FKSZ18}. 
For $\nu \in \{\alpha,\beta,\gamma\}$, let $L_\nu$ be the subgroup of $H_1(\Sigma;\Z)$ generated by the homology classes of the curves $\nu_i$. 
 We introduce the following complex $C$:
\begin{equation*} 
 0 \to \Z \fl{\scriptstyle{0}} (L_\alpha \cap L_\beta)\oplus(L_\beta \cap L_\gamma)\oplus(L_\gamma \cap L_\alpha)
 \fl{\zeta} L_\alpha \oplus L_\beta \oplus L_\gamma \fl{\iota} H_1(\Sigma) \fl{\scriptstyle{0}} \Z \to 0,
 \end{equation*}
 where $\zeta(x,y,z)=(x-z,y-x,z-y)$ and $\iota$ is defined by the inclusions $L_\nu \hookrightarrow H_1(\Sigma)$ for $\nu\in\{\alpha,\beta,\gamma\}$.

\begin{theorem} \label{thhomologyXnotwist} 
 The homology of $X$ with coefficient in $\mathbb{Z}$ canonically identifies with the homology of the complex $C$. 
 In particular:
 \begin{equation*}
  H_1(X)  \simeq \frac{H_1(\Sigma)}{L_\alpha+L_\beta+L_\gamma} ; \
  H_2(X)  \simeq \frac{L_\alpha\cap(L_\beta+ L_\gamma)}{(L_\alpha\cap L_\beta)+(L_\alpha\cap L_\gamma)}; \
  H_3(X) \simeq L_\alpha\cap L_\beta\cap L_\gamma.
 \end{equation*}
\end{theorem}
The spaces in the complex $C$ can be understood as spaces of chains in $X$. This will be made explicit in Section \ref{secintform}. 
Note that $C$ is not the same as the complex considered in  \cite{FKSZ18}: ours is symmetric in the three subspaces $L_\alpha,L_\beta$ and $L_\gamma$. Moreover, the expression of $H_1(\Sigma;\Z)$ was not provided in \cite{FKSZ18}.

\begin{example} \label{ex}
Consider the trisection diagram of $S^2 \times S^2$ given in Figure \ref{fig}, where the curves $\alpha$ are in red, the curves $\beta$ in blue and the curves $\gamma$ in green.
One easily checks that $L_\alpha \cap L_\beta \cap L_\gamma = \{ 0 \}$ and $L_\alpha+L_\beta+L_\gamma\simeq H_1(\Sigma;\Z)$, giving $H_1(S^2 \times S^2)=0=H_3(S^2 \times S^2)$. Similarly, $H_2(S^2 \times S^2)\simeq L_\gamma \simeq \Z^2$.
\end{example}

Now fix a non-trivial morphism $\varphi: H_1(X;\Z) \rightarrow G$, where $G$ is a finitely generated free abelian group. Thanks to Theorem \ref{thhomologyXnotwist}, it induces a morphism $H_1(\Sigma;\Z) \rightarrow G$, still denoted $\varphi$. 
Let $\F=\Q(G)$ be the quotient field of the group ring $\Lambda=\Z[G]$. Let $R$ stand for $\Lambda$ or $\F$. For $\nu\in\{\alpha,\beta,\gamma\}$, let $L_\nu^\pr$ be the subspace of $H_1^\varphi(\Sigma;R)$ generated by the $R$--homology classes of the curves $\nu_i$. Let $C^\pr$ be the following complex: %es over $\F$ and $\Lambda$ respectively:
%\begin{equation*} \label{complex}
% 0 \to (L_\alpha^\pf \cap L_\beta^\pf)\oplus(L_\beta^\pf \cap L_\gamma^\pf)\oplus(L_\gamma^\pf \cap L_\alpha^\pf) \fl{\zeta} L_\alpha^\pf \oplus L_\beta^\pf \oplus L_\gamma^\pf \fl{\iota} H_1^\varphi(\Sigma;\F) \to 0,
%\end{equation*}
\begin{equation*} %\label{complexLambda}
 0\hspace{-1.2pt} \to (L_\alpha^\pr \cap L_\beta^\pr)\oplus(L_\beta^\pr \cap L_\gamma^\pr)\oplus(L_\gamma^\pr \cap L_\alpha^\pr)
 \fl{\zeta} L_\alpha^\pr \oplus L_\beta^\pr \oplus L_\gamma^\pr \fl{\iota} H_1^\varphi(\Sigma;R) \fl{0} H_0^\varphi(X;R) \to \hspace{-1.2pt} 0,
\end{equation*}
where $\zeta(x,y,z)=(x-z,y-x,z-y)$ and $\iota$ is defined by the inclusions $L_\nu^\pr\hookrightarrow H_1^\varphi(\Sigma;R)$.
Note that, with coefficients in $\F$, we have $H_0^\varphi(X;\F)=0$.

\begin{theorem} \label{thtwistedhomologyX}
 The homology of $(X;\varphi)$ with coefficients in $R$ canonically identifies with the homology of the complex $C^\pr$.
 In particular, with coefficients in $R$:
 $$H_1^\varphi(X)  \simeq \frac{H_1^\varphi(\Sigma)}{L_\alpha^\pr+L_\beta^\pr+L_\gamma^\pr}, \ \,
  H_2^\varphi(X) 
    \simeq \frac{L^\pr_\alpha \cap(L^\pr_\beta + L^\pr_\gamma)}{(L^\pr_\alpha \cap L^\pr_\beta)+ (L^\pr_\alpha \cap L^\pr_\gamma)}, \ \, 
  H_3^\varphi(X)  \simeq L_\alpha^\pr \cap L_\beta^\pr \cap L_\gamma^\pr.$$
\end{theorem}

This result provides in particular an expression of the Alexander module $H_1^\varphi(X;\Lambda)$. However the $\Lambda$--module $H_1^\varphi(\Sigma;\Lambda)$ and its submodules $L_\nu^\pl$ are not free modules in general, so that we do not get a free presentation of the Alexander module. However, one can compute the Alexander polynomial of $(X;\varphi)$ using the following trick. Let $B$ be a 4--ball in $X$ that intersects $\Sigma$ transversely along a disk $D$ disjoint from the $3g$ curves of the diagram. Set $\p X=X\setminus\Int B$ and $\p\Sigma=\Sigma\setminus\Int D$. Fix a base-point $\star\in\partial D$. 
One easily checks that the $\Lambda$--modules $H_1^\varphi(X;\Lambda)$ and $H_1^\varphi(\p X,\star;\Lambda)$ have the same $\Lambda$--torsion submodule, so that $(X;\varphi)$ and $(\p X,\star;\varphi)$ have the same Alexander polynomial.

For $\nu\in\{\alpha,\beta,\gamma\}$, let $\p L_\nu^\pl$ be the subspace of $H_1^\varphi(\p\Sigma,\star;\Lambda)$ generated by the homology classes of the curves $\nu_i$. 
We show in Lemma \ref{lemmaTorsXpX} that $H_1^\varphi(\p\Sigma,\star;\Lambda)$, $\p L^\pl_\nu$ and $\p L^\pl_\nu\cap \p L^\pl_{\nu'}$ are free $\Lambda$--modules. 
As previously, we consider a complex of $\Lambda$--modules $\p C^\pl$: 
$$ 0 \to (\p L^\pl_\alpha \cap \p L^\pl_\beta)\oplus(\p L^\pl_\beta \cap \p L^\pl_\gamma)\oplus(\p L^\pl_\gamma \cap \p L^\pl_\alpha)
 \fl{\zeta} \p L^\pl_\alpha \oplus \p L^\pl_\beta \oplus \p L^\pl_\gamma  \fl{\iota} H_1^\varphi(\p\Sigma,\star;\Lambda) \to 0.$$
With the very same proof as for Theorem \ref{thtwistedhomologyX}, one shows the following result.

\begin{theorem} \label{thAlexander}
 The $\Lambda$--homology of $(\p X,\star;\varphi)$ canonically identifies with the homology of the complex $\p C^\pl$. In particular, the Alexander $\Lambda$--module of $(\p X,\star;\varphi)$ admits the finite presentation:
 $$\p L_\alpha^\pl\oplus\p L_\beta^\pl\oplus \p L_\gamma^\pl \to H_1^\varphi(\p \Sigma,\star;\Lambda) \to H_1^\varphi(\p X,\star;\Lambda)  \to 0.$$
\end{theorem}
This result provides a presentation matrix of the Alexander $\Lambda$--module of $(\p X,\star;\varphi)$ from which one can compute the Alexander polynomial of $(\p X,\star;\varphi)$ and $(X;\varphi)$.

\subsection{Intersection forms}

For $R=\Lambda$ or $\F$, we express the intersection form of $(X;\varphi)$ using the expression of $H_2^\varphi(X;R)$ given by Theorem~\ref{thtwistedhomologyX} and the intersection form of $(\Sigma;\varphi)$. The spaces $L_\nu^\pr$ coincide with $$\ker\left(H_1^\varphi(\Sigma;R)\fl{\scriptstyle{incl_*}} H_1^\varphi(H_\nu;R)\right).$$
%They are Lagrangian subspaces with respect to the intersection form $\langle\cdot,\cdot\rangle_{\Sigma}^\varphi$ of $(\Sigma;\varphi)$.
Define the hermitian form
 $$\lambda^\varphi: \frac{L^\pr_\alpha \cap (L^\pr_\beta + L^\pr_\gamma)}{(L^\pr_\alpha\cap L^\pr_\beta)+ (L^\pr_\alpha\cap L^\pr_\gamma)} 
 \times \frac{L^\pr_\alpha \cap (L^\pr_\beta + L^\pr_\gamma)}{ (L^\pr_\alpha \cap L^\pr_\beta)+ (L^\pr_\alpha \cap L^\pr_\gamma)} \longrightarrow R$$
as follows. For $a,a'  \in L_\alpha^\pr \cap (L_\beta^\pr + L_\gamma^\pr)$ and $b\in L_\beta^\pr$, $c\in L_\gamma^\pr$ such that $a+b+c=0$, 
set $$ \lambda^\varphi(a,a'):= \langle c, a' \rangle^\varphi_{\Sigma}.$$ 
Note that permuting the roles of $\alpha$, $\beta$ and $\gamma$ in this construction gives the same form, up to the sign of the permutation. 
This is related to the fact that the coorientation of $\Sigma$ ---defined by the orientations of $\Sigma$ and $X$--- induces a cyclic order on the $X_i$ and the~$H_\nu$. 
\begin{theorem} \label{thtwistedintform} 
Let $\langle \cdot,\cdot \rangle_{X}^\varphi$ be the intersection form of $(X;\varphi)$.
There is an isomorphism
$$ \left(H_2^\varphi(X;R), \langle \cdot,\cdot \rangle_{X}^\varphi \right) \simeq \left(\frac{L^\pr_\alpha \cap 
(L^\pr_\beta + L^\pr_\gamma)}{ (L^\pr_\alpha \cap L^\pr_\beta)+ (L^\pr_\alpha \cap L^\pr_\gamma)} , \lambda^\varphi \right).$$
\end{theorem}

The form $\lambda^\varphi$ is a hermitian version of a symmetric form introduced by Wall in \cite{Wa69}, which is involved in the similar result in the non-twisted setting \cite[Theorem 3.6]{FKSZ18}, corresponding to a trivial morphism $\varphi$. As noted in \cite[Remark~3.7]{FKSZ18}, the main theorem of \cite{Wa69} implies that the signature of the intersection form of $(X;\varphi)$ equals the signature of the form $\lambda^\varphi$. In the case of trisections, the above theorem says that not only the signatures coincide, but also the forms themselves.

\begin{proposition} \label{propintfromH13twisted}
 There is an isomorphism
 $$ \left(H_1^\varphi(X;R)\times H_3^\varphi(X;R); \langle \cdot,\cdot \rangle_{X}^\varphi \right) \simeq 
 \left(\frac{H_1^\varphi(\Sigma)}{L^\pr_\alpha + L^\pr_\beta + L^\pr_\gamma}\times(L^\pr_\alpha\cap L^\pr_\beta\cap L^\pr_\gamma); 
 \langle\cdot,\cdot\rangle^\varphi_\Sigma \right).$$
\end{proposition}

\subsection{Abelian torsions}

We now state the result for the torsion. Consider the complex $C^\pf$ defined before Theorem \ref{thtwistedhomologyX}.

\begin{theorem} \label{thtorsion}
 There exists an $\F$-basis $c$ for $C^\pf$ such that for any homology $\F$-basis $h$ of $X$ and $C^\pf$, the following holds:
 $$\tau^\varphi(X;h)=\tau(C^\pf;c,h) \quad\textrm{in }\F/\pm\varphi(H_1(X)).$$
\end{theorem}

The complex basis $c$ is explicited in Subsection \ref{subsecbases}. Although the bases for $H_1^\varphi(\Sigma;\F)$ and the $L^\pf_\nu$ are straightforwardly obtained from the trisection diagram, the computation of the bases for the intersections $L^\pf_\nu\cap L^\pf_{\nu'}$ involves handleslides on the surface. From an algorithmic point of view, this might not be efficient. As an alternative way, one may use the same trick as in Section \ref{sechomol} and compute the torsion of $(\p X, \star)$ instead, where $\p X$ is the complement of $4$-ball and $\star$ is a base point in the boundary. The two torsions $\tau^\varphi(X)$ and $\tau^\varphi(\p X, \star)$ coincide up to a factor, see Proposition \ref{propTorsXpX}. This allows to use much more general complex bases, avoiding the handleslides, see Theorem \ref{thtorsionXstar}. The (light) price to pay is that $\tau^\varphi(\p X, \star)$ is computed only up to a unit in $\Lambda$.

\section{Preliminaries} \label{secprelim}

\subsection{Algebraic torsion}

We recall the algebraic setup, see \cite{Mi66} and \cite{Tu01} for further details and references. Let $\K$ be a field. If $V$ is a finite dimensional $\K$--vector space and $b$ and $c$ are two bases of $V$, we denote by $[b/c]$ 
the determinant of the matrix expressing the basis change from $b$ to $c$. The bases $b$ and $c$ are {\em equivalent} if $[b/c]=1$. Let $C$ be a finite complex of finite dimensional $\K$--vector spaces:
$$
C=( C_m\ \fl{\partial_{m}}\ C_{m-1}\ \longrightarrow  \ \cdots\ \fl{\partial_1}\ C_0 ).
$$
A \emph{complex basis} of $C$ is a family  $c=(c_m,\dots,c_0)$ where $c_i$ is a basis of  $C_i$ for all $i\in\{0,\dots,m\}$. A \emph{homology basis} of $C$ is a family $h=(h_m,\dots,h_0)$ where $h_i$ is a basis of the homology group $H_i(C)$ for all $i\in\{0,\dots,m\}$.
If we have chosen a basis $b_j$ of the space of $j$--dimensional boundaries $B_j(C):= \operatorname{Im} \partial_{j+1}$ for all $j\in \{0,\dots,m-1\}$, then a homology basis $h$ of $C$ induces an equivalence class of bases $(b_ih_i)b_{i-1}$ of $C_i$ for all $i$. 

The \emph{torsion} of the $\K$--complex $C$, equipped with a complex basis $c$ and a homology basis $h$, is the scalar:
\begin{displaymath} 
\tau(C;c,h):=  \prod_{i=0}^m \big[(b_ih_i)b_{i-1}/c_i\big]^{(-1)^{i+1}} \in \K^*.
\end{displaymath}
It is easily checked that this definition does not depend on the choice of $b_0,\dots,b_m$. When $C$ is acyclic, we set $\tau(C;c):= \tau(C;c,\varnothing)$.
\begin{lemma}\label{lem:torsion_as_function}
Consider a short exact sequence $0 \rightarrow C'  \rightarrow C \rightarrow  C''  \rightarrow 0$ of $\K$--complexes with compatible complex bases $c',c,c''$ in the sense that $c_i$ is equivalent to $c'_ic''_i$ for every $i\in \{0,\dots,m\}$ and homology bases $h',h,h''$. 
The associated long exact sequence in homology $\mathcal{H}$ is an acyclic finite $\K$--complex with base $(h',h,h''):= (h'_m,h_m,h''_m,\dots,h'_0,h_0,h''_0)$ and we have
$$ \tau(C;c,h)=  \varepsilon \cdot  \tau(C';c',h') \cdot \tau(C'';c'',h'') \cdot \tau\big(\mathcal{H};  (h',h,h'')\big),$$
where $\varepsilon$ is a sign  depending on the dimensions of $C'_i$, $C_i$, $C''_i$ and $H_i(C')$, $H_i(C)$, $H_i(C'')$ for $i\in \{0,\dots,m\}$.
\end{lemma}

\subsection{Twisted homology and Reidemeister torsion} 

Let $(X,Y)$ be a finite CW--pair with maximal abelian cover $p: \overline{X} \rightarrow X$.
Let $G$ be a finitely generated free abelian group. Fix a group homomorphism $\varphi: \pi_1(X) \rightarrow G$ and 
denote $R$ the group ring $\Lambda=\Z[G]$ or its quotient field $\F=\Q(G)$. The extension of $\varphi$ to a ring morphism $\Z[H_1(X)] \rightarrow R$ is still denoted $\varphi$.
The chain complex of $(X,Y;\varphi)$ with coefficient in $R$ is defined as 
$$C^{\varphi}(X,Y; R)= C(\overline{X},p^{-1}(Y))\otimes_{\Z[H_1(X)] } R.$$ 
We denote $H^\varphi(X,Y; R)$ its homology. 
It is easy to check that
$$ H^\varphi(X,Y; \F) = H^\varphi(X,Y; \Lambda) \otimes_{\Lambda} \F.$$

Let $\overline{c}$ be a complex basis of the complex of free $\Z[H_1(X)]$--module $C(\overline{X},p^{-1}(Y))$ obtained by lifting each relative cell of $(X,Y)$ to $\overline{X}$. Then $c=\overline{c}\otimes 1$ is a complex basis of $C^{\varphi}(X,Y;\F)$. 
\begin{definition}
Given a homology basis $h$ of $H^\varphi(X,Y; \F)$, the torsion of $(X,Y;\varphi)$ is 
$$ \tau^\varphi(X,Y;h) := \tau(C^{\varphi}(X,Y; \F);c,h) \in \F /\pm \varphi(H_1(X)).$$
\end{definition}

The ambiguity in $\pm \varphi(H_1(X))$ is due to the different choices of lift and orientation of the cells. 
Note that the torsion of $(X,Y;\varphi)$ is closely related to the orders of the $\Lambda$--modules $H^\varphi(X,Y;\Lambda)$, see \cite{KL}.

We end the subsection with two useful results.
\begin{lemma} \label{lemmaH0}
If $X$ is connected and $\varphi$ is non-trivial, then $H_0^\varphi(X;\F)=0$.
\end{lemma}
By Blanchfield duality (see Subsection \ref{sectionbl}), Lemma \ref{lemmaH0} implies the following corollary.
\begin{corollary} \label{corHn}
 Assume $X$ is a compact connected oriented $n$--manifold. 
 If $\varphi$ is non-trivial, then $H_n^\varphi(X,\partial X;\F)=0$.
\end{corollary}

\subsection{Twisted intersection form and Blanchfield duality} \label{sectionbl}

Let $W$ be a compact oriented $n$--manifold and $\varphi:\pi_1(W) \rightarrow G$ be a group homomorphism. 
For $q \in \{0,\dots,n \}$, the \emph{twisted intersection form of $W$} with coefficient in $R$, introduced by Reidemeister in \cite{Re39}, is the sesquilinear map
\begin{equation*} \label{equivform}
 \langle\cdot,\cdot\rangle_W^\varphi: H_q^\varphi(W;R) \times H_{n-q}^\varphi(W,\partial W;R) \longrightarrow R 
\end{equation*}
defined by 
$$ \langle x \otimes z , x' \otimes z' \rangle_W^\varphi= \sum_{h \in \frac{H_1(W)}{\ker(\varphi)}} \langle x, h.x'\rangle_{\overline W} \, \varphi(h)\otimes zz',$$
where $R=\Lambda$ or $\F$, $\overline{W}\twoheadrightarrow W$ is the covering associated with $\ker(\varphi)$ and $\langle\cdot,\cdot\rangle_{\overline W}$ stands for the algebraic intersection in $\overline{W}$. 
By Blanchfield's duality theorem \cite[Theorem 2.6]{Bl57}, for $R=\F$, this form is nondegenerate.

\subsection{Reconstruction of $X$ from the trisection diagram} \label{subsecReconstruction}

A trisection diagram determines the associated 4--manifold. We recall here how to reconstruct the manifold from the diagram and we introduce some notations that we will use in the next sections.

We are given a trisection diagram, {\em i.e.} a closed genus $g$ surface $\Sigma$ and three systems of meridians $(\alpha_i)_{1\leq i\leq g}$, 
$(\beta_i)_{1\leq i\leq g}$, and $(\gamma_i)_{1\leq i\leq g}$. Consider a disk $D^2$ with center $p_0$ and three distinct points $p_\alpha$, $p_\beta$, 
$p_\gamma$ on the boundary, see Figure \ref{figreconstructX}. Take the product $D^2\times\Sigma$ and add a 2--cell along $p_\nu\times\nu_i$ for all $\nu=\alpha,\beta,\gamma$ and all $1\leq i\leq g$. 
It remains to add 3--cells and 4--cells; the way this is performed as no incidence, see Laudenbach and Po\'enaru \cite{LP}. 
\begin{figure}
\begin{center}
\begin{tikzpicture}
 \foreach \t in {1,2,3} {
  \draw[rotate=120*\t] (0.8,1) node {$X_\t$};
 % \draw[rotate=120*\t,color=red,line width=6pt] (0,0) -- (2,0);
  \draw[rotate=120*\t] (0,0) -- (2,0);}
 \draw (0,0) circle (2);
 \draw (2.5,0) node {$H_\gamma$};
 \draw[rotate=120] (2.5,0) node {$H_\alpha$};
 \draw[rotate=240] (2.5,0) node {$H_\beta$};
 \draw (0,0) node {$\scriptstyle{\bullet}$} node[above right] {$\Sigma$};
 %\draw[color=red] (1.6,0.4) node {$N$};
 \draw[blue] (2,0) node {$\scriptstyle{\bullet}$} node[below right] {$p_\gamma$};
 \draw[rotate=120,blue] (2,0) node {$\scriptstyle{\bullet}$} (1.9,-0.1) node[right] {$p_\alpha$};
 \draw[rotate=240,blue] (2,0) node {$\scriptstyle{\bullet}$} (1.9,0.1) node[right] {$p_\beta$};
 \draw[blue] (0,0) node {$\scriptstyle{\bullet}$} (0,-0.1) node[below right] {$p_{\scriptstyle{0}}$};
\end{tikzpicture}
\end{center} \caption{Decomposition of $X$.} \label{figreconstructX}
\end{figure}
In this decomposition of $X$, $H_\nu$ is recovered as $[p_0,p_\nu]\times\Sigma$ union the corresponding 2--cells and 3--cell. The union $Y=H_\alpha\cup H_\beta\cup H_\gamma$ is the {\em spine} of the trisection. In the computations of homology and torsion, we will make a great use of the exact sequences in homology associated with the pairs $(Y,\Sigma)$ and $(X,Y)$.

\section{Homology of $X$} \label{sechomologyX}

Throughout this section, $H_\ell(.)$ stands for the homology with coefficients in $\Z$. 
We fix a trisection $X=X_1 \cup X_2 \cup X_3$ given by a diagram $(\Sigma;\alpha,\beta,\gamma)$, with spine $Y=H_\alpha \cup H_\beta \cup H_\gamma$ (see Section 3.4). We prove Theorem \ref{thhomologyXnotwist} using the pairs $(Y,\Sigma)$ and $(X,Y)$.

\begin{lemma} \label{lemmaYShom}
 $$H_\ell(Y,\Sigma)\simeq\left\lbrace\begin{array}{c c} L_\alpha\oplus L_\beta\oplus L_\gamma & \textrm{if }\ell=2 \\ \Z^3 & \textrm{if }\ell=3 \\ 0 & \textrm{otherwise} \end{array}\right.$$
\end{lemma}
\begin{proof}
 Note that $H_\ell(Y,\Sigma)\simeq H_\ell(H_\alpha,\Sigma)\oplus H_\ell(H_\beta,\Sigma)\oplus H_\ell(H_\gamma,\Sigma)$ and:
 $$H_\ell(H_\nu,\Sigma)\simeq\left\lbrace\begin{array}{c c} L_\nu & \textrm{if }\ell=2 \\ \Z & \textrm{if }\ell=3 \\ 0 & \textrm{otherwise} \end{array}\right.,$$ 
 thanks to the exact sequence associated with $(H_\nu,\Sigma)$.
\end{proof}

\begin{lemma} \label{lemmaYhom}
The homology of $Y$ with coefficients in $\Z$ is given by
 \begin{equation*}
  H_1(Y)  \simeq \frac{H_1(\Sigma)}{L_\alpha+L_\beta+L_\gamma}; \
  H_2(Y)  \simeq \ker\left(L_\alpha\oplus L_\beta\oplus L_\gamma\to H_1(\Sigma)\right); \
  H_3(Y)  \simeq \Z^2.
 \end{equation*}
\end{lemma}
\begin{proof}
 We use the exact sequence associated with the pair $(Y,\Sigma)$. 
 Since the surface $\Sigma$ bounds any of the $H_\nu$ in $Y$, the map $H_2(\Sigma)\to H_2(Y)$ is trivial and we get:
 $$0 \to H_3(Y) \to H_3(Y,\Sigma) \to H_2(\Sigma) \to 0,$$
 from which we easily see that $H_3(Y)\simeq\Z^2$, generated by two of the boundaries $\partial X_i$, and:
 $$0 \to H_2(Y) \to H_2(Y,\Sigma) \to H_1(\Sigma) \to H_1(Y) \to 0,$$
 which provides the expressions of $H_1(Y)$ and $H_2(Y)$, using Lemma \ref{lemmaYShom}.
\end{proof}

\begin{lemma} \label{lemmaXYhom} 
 $$H_\ell(X,Y)\simeq\left\lbrace\begin{array}{c c} (L_\alpha\cap L_\beta)\oplus(L_\beta\cap L_\gamma)\oplus(L_\gamma\cap L_\alpha) & \textrm{if }\ell=3 \\ \Z^3 & \textrm{if }\ell=4 \\ 0 & \textrm{otherwise} \end{array}\right.$$
\end{lemma}
\begin{proof} 
 Note that $H_\ell(X,Y)\simeq \oplus_{j=1}^3 H_\ell(X_j,\partial X_j)$. One easily checks that $H_4(X_j,\partial X_j)\simeq\Z$ and $H_\ell(X_j,\partial X_j)=0$ if $\ell\neq3,4$. For $\ell=3$, the exact sequence associated with $(X_j,\partial X_j)$ gives $H_3(X_j,\partial X_j)\simeq H_2(\partial X_j)$. Now, for $\nu,\nu' \in \{\alpha,\beta,\gamma \}$ such that $\partial X_j=H_\nu \cup_{\Sigma} H_{\nu'}$, the Mayer--Vietoris sequence gives:
 $$0\to H_2(\partial X_j)\to H_1(\Sigma)\to H_1(H_\nu)\oplus H_1(H_{\nu'}).$$
 Hence $H_2(\partial X_j)\simeq L_\nu\cap L_{\nu'}$. 
\end{proof}

\begin{proof}[Proof of Theorem  \ref{thhomologyXnotwist}]
 We use the exact sequence associated with the pair $(X,Y)$. The map $H_3(Y)\to H_3(X)$ is trivial since $H_3(Y)$ is generated by the classes of the $\partial X_j$ for $j \in \{1,2,3\}$, which bound the $X_j$ in $X$. Thus
 $$0\to H_3(X)\to H_3(X,Y)\fl{\zeta} H_2(Y)\to H_2(X)\to 0$$
 and $H_1(Y)\simeq H_1(X)$, thanks to Lemma \ref{lemmaXYhom}. The expression of the map $\zeta$ follows of the descriptions of $H_2(Y)$ in Lemma \ref{lemmaYhom} and of $H_3(X,Y)$ in 
 Lemma \ref{lemmaXYhom}.
\end{proof}

We end this section with a lemma which will be useful in the next sections.
\begin{lemma}
 If $G$ is a finitely generated abelian group and $\varphi: H_1(X) \rightarrow G$ a non-trivial morphism, then $\varphi$ induces non-trivial morphisms 
 on the first homology groups of the spaces $\Sigma$, $H_\nu$ for $\nu \in \{ \alpha,\beta ,\gamma \}$, $Y$, $X_i$ and $\partial X_i$ for $i=1,2,3$.
\end{lemma}
\begin{proof}
 For $Y$, it is obvious since $H_1(Y)\simeq H_1(X)$. Thanks to the expression of $H_1(X)$ given in Theorem \ref{thhomologyXnotwist}, the homology groups $H_1(\Sigma)$, $H_1(H_\nu)\simeq H_1(\Sigma)/L_\nu$ and $H_1(X_1)\simeq H_1(\partial X_1)\simeq H_1(\Sigma)/(L_\alpha\cap L_\beta)$ naturally surject on $H_1(X)$. Compose $\varphi$ by these surjections. 
\end{proof}
For simplicity all these morphisms induced by $\varphi$ are still denoted $\varphi$.

\section{Homology of $(X;\varphi)$} \label{sechomologyXtwist}

In this section, we compute the homology of $(X;\varphi)$ for a non-trivial $\varphi$. As in Section~\ref{sechomologyX}, we use the pairs $(Y,\Sigma)$ and $(X,Y)$. The notation $R$ is used for either $\Lambda$ or $\F$ and $H^\varphi(\cdot)$ stands for the homology with coefficients in $R$. 

\begin{lemma}
 For any $\nu\in\{\alpha,\beta,\gamma\}$, $L_\nu^\pr=\ker\left(H_1^\varphi(\Sigma)\fl{\scriptstyle{incl_*}} H_1^\varphi(H_\nu)\right)$.
\end{lemma}
\begin{proof}
 Let $\nu\in\{\alpha,\beta,\gamma\}$. Since $H_\nu$ retracts on a wedge of circles, we have $H_2^\varphi(H_\nu)=0$ and the exact sequence in homology of the pair $(H_\nu,\Sigma)$ provides the short exact sequence:
 $$0\to H_2^\varphi(H_\nu,\Sigma)\to H_1^\varphi(\Sigma)\to H_1^\varphi(H_\nu).$$
 Now, $H_\nu$ is obtained from $\Sigma$ by adding meridian disks $D_i^\nu$ such that $\partial D_i^\nu=\nu_i$ and one 3--cell. Hence $H_2^\varphi(H_\nu,\Sigma)$ is generated by the $D_i^\nu$ for $i=1,\dots,g$ and its image in $H_1^\varphi(\Sigma)$ is exactly the submodule $L_\nu^\pr$ generated by the $\nu_i$.
\end{proof}

\begin{lemma} \label{lemmaYShomtwist}
 We have $H^\varphi_\ell(Y,\Sigma)=0$ for $\ell\neq2$. Moreover, there is a natural identification
 $$H^\varphi_2(Y,\Sigma)\simeq L_\alpha^\pr\oplus L_\beta^\pr\oplus L_\gamma^\pr.$$
\end{lemma}
\begin{proof}
 Use $H^\varphi_\ell(Y,\Sigma)\simeq H^\varphi_\ell(H_\alpha,\Sigma)\oplus H^\varphi_\ell(H_\beta,\Sigma)\oplus H^\varphi_\ell(H_\gamma,\Sigma)$ and the exact sequence associated with $(H_\nu,\Sigma)$.
\end{proof}

\begin{lemma} \label{lemmaYhomtwist}
 There are natural identifications:
 \begin{align*}
  & H_1^\varphi(Y) \simeq \frac{H_1^\varphi(\Sigma)}{L_\alpha^\pr+L_\beta^\pr+L_\gamma^\pr}, \\
  & H_2^\varphi(Y) \simeq \ker\, \big( L_\alpha^\pr\oplus L_\beta^\pr\oplus L_\gamma^\pr\to H_1^\varphi(\Sigma) \big).
 \end{align*}
\end{lemma}
\begin{proof}
 Thanks to Lemma \ref{lemmaYShomtwist}, the exact sequence in homology of the pair $(Y,\Sigma)$ reduces to   
 \begin{equation*} 
  0\to H_2^\varphi({Y})\longrightarrow L_\alpha^\pr\oplus L_\beta^\pr\oplus L_\gamma^\pr \to H_1^\varphi(\Sigma)\longrightarrow H_1^\varphi({Y})
  \longrightarrow 0.
 \end{equation*}
 This provides the given expressions for the homology of $Y$.  
\end{proof}

\begin{lemma} \label{lemmaXYhomtwist}
 We have $H^\varphi_\ell(X,Y)=0$ for $\ell\neq3$. Moreover, there is a natural identification
 $$H^\varphi_3(Y,\Sigma)\simeq(L_\alpha^\pr\cap L_\beta^\pr)\oplus(L_\beta^\pr\cap L_\gamma^\pr)\oplus(L_\gamma^\pr\cap L_\alpha^\pr).$$ 
\end{lemma}
\begin{proof}
 Note that $H^\varphi_\ell(X,Y)\simeq \oplus_{j=1}^3 H^\varphi_\ell(X_j,\partial X_j)$. Let's focuse on $(X_1,\partial X_1)$. For $\ell=3$, the exact sequence associated with $(X_1,\partial X_1)$ gives $H_3(X_1,\partial X_1)\simeq H_2(\partial X_1)$. Now, the Mayer--Vietoris sequence associated with the Heegaard splitting ${\partial {X}}_1=H_\alpha\cup_\Sigma H_\beta$ gives:
 \begin{equation*}
 0\to H_2^\varphi({\partial {X}}_1)\to H_1^\varphi(\Sigma )\fl{f} H_1^\varphi(H_\alpha)\oplus H_1^\varphi(H_\beta).
 \end{equation*}
 We get $H_2^\varphi({\partial {X}}_1)\simeq\ker(f)=L_\alpha^\pr\cap L_\beta^\pr$.
\end{proof}

\begin{proof}[Proof of Theorem \ref{thtwistedhomologyX}]
 The exact sequence associated with the pair $(X,Y)$ gives
 $$0\to H^\varphi_3(X)\to H^\varphi_3(X,Y)\fl{\zeta} H^\varphi_2(Y)\to H^\varphi_2(X)\to 0$$
 and $H^\varphi_1(X)\simeq H^\varphi_1(Y)$. 
 The result then follows from Lemmas~\ref{lemmaYhomtwist} and~\ref{lemmaXYhomtwist}.
\end{proof}

\section{Torsion of $(X;\varphi)$} \label{secTorsion}

In this section, $H^\varphi(\cdot)$ stands for the twisted homology with coefficients in $\F$ and we assume that $\varphi$ is non-trivial. The aim of the section is to prove Theorem \ref{thtorsion}.
We fix a lift of a given $0$--cell $\star$ of $X$ and require that the lift of any $1$--cell starts at the chosen lift of $\star$.

\subsection{Bases in homology and first computations} \label{subsecbases}

In this subsection, we define bases for the complex $C^\pf$ of Theorem \ref{thtorsion} and we compute some related torsions that we need to prove the theorem.

\begin{definition} \label{defsympl}
A {\em (geometric) symplectic basis of $\Sigma$} is a family $(x_i,y_i)_{1\leq i\leq g}$ of simple closed curves in $\Sigma$, based at $\star$, such that
\begin{itemize}
\item any two curves meet only at $\star$,
\item the classes of $x_i$ and $y_i$ form a basis of $H_1(\Sigma;\Z)$,  symplectic with respect to the intersection form $\langle .,. \rangle_\Sigma$,
\item there exists a CW--complex decomposition of $\Sigma$ with a single 2--cell glued along $$\partial\Sigma=\prod_{i=1}^g[x_i,y_i].$$
\end{itemize}
\end{definition}

\begin{lemma} \label{lemmaBaseSigma}
Let $(x_i,y_i)_{1\leq i\leq g}$ be a symplectic basis of $\Sigma$ such that $x_i\in\ker(\varphi)$ for all $i$.  Up to re-ordering, assume that $y_1\notin\ker(\varphi)$. 
For $i=2,\dots,g$, set
$$\ y'_i= y_i-\frac{\varphi(y_i)-1}{\varphi(y_1)-1}\, y_1 \in C_1^\varphi(\Sigma).$$
 Then the family $\hS=(x_i,y'_i)_{i>1}$ is a basis of $H_1^\varphi(\Sigma)$ and $\tau^\varphi(\Sigma;\hS)= 1$.
\end{lemma}

For instance, in Lemma \ref{lemmaBaseSigma}, one can choose the family $(x_i)_i$ to coincide with the family $\alpha$, $\beta$ or $\gamma$ and $(y_i)_i=(x_i^*)_i$ to be a dual family. Note that $H_1^\varphi(\Sigma)=0$ if $g=1$.
\begin{proof}
 There is a CW--decomposition of $\Sigma$ given by $\star$ as 0--cell, the $x_i$ and $y_i$ as 1--cell and $\Sigma$ as 2--cell. The associated $\F$-complex is 
 $C^\varphi(\Sigma): \ 0\to C_2^\varphi(\Sigma)\to C_1^\varphi(\Sigma)\to C_0^\varphi(\Sigma)\to 0$
 with basis $\Sigma$, $(x_i,y_i)_i$ and $\star$. We choose the lift of $\Sigma$ so that  %$\partial\tilde\Sigma=\prod_{i=1}^g[x_i,y_i]$, 
 $$\partial \Sigma=\sum_{1\leq i \leq g}(\varphi(y_i)-1)\, x_i\ \in C_1^\varphi(\Sigma).$$
 For all $i$, $\partial x_i=0$ and $\partial y_i=(\varphi(y_i)-1)\,\star$. Hence $\hS$ is a basis of $H_1^\varphi(\Sigma)$. We get
 $$\tau^\varphi(\Sigma;\hS)=\left[\frac{\partial \Sigma\cdot \hS\cdot (\varphi(y_1)-1)^{-1} y_1}{(x_i,y_i)_{1\leq i\leq g}}\right] =1.$$
\end{proof}

\begin{lemma} \label{lemmaBaseLag}
Let $\nu\in\{\alpha,\beta,\gamma\}$ and  $(\nu_i^*)_{1\leq i\leq g}$ be simple closed curves in $\Sigma$ such that $(\nu_i,\nu_i^*)_{1\leq i\leq g}$ is a symplectic basis for $\Sigma$. Permuting the indices if necessary, assume that $\varphi(\nu_1^*)\neq1$. 
 Then the family $h_\nu=(\frac{1}{\varphi(\nu_1^*)-1}\, \nu_2, \nu_3,\dots,\nu_g)$ is a basis of $L_\nu^\pf$.
\end{lemma}
\begin{proof}
 By definition, $L_\nu^\pf$ is generated by the $\nu_i$. Consider the same complex as in the proof of Lemma \ref{lemmaBaseSigma}, with $x_i=\nu_i$ and $y_i=\nu_i^*$. The only relation is \mbox{$\sum_{1\leq i \leq g}(\varphi(\nu_i^*)-1)\, \nu_i=0$.}
\end{proof}

\begin{lemma} \label{lemmaBaseYSigma}
 Via the identification $H_2^\varphi(Y,\Sigma)\simeq L_\alpha^\pf\oplus L_\beta^\pf\oplus L_\gamma^\pf$, the family $\hYS=\ha.\hb.\hc$ is a homology basis of $(Y,\Sigma;\varphi)$ and we have $\tau^\varphi(Y,\Sigma;\hYS)=1$.
\end{lemma}
\begin{proof}
 By Lemma \ref{lemmaYShomtwist}, $H_2^\varphi(Y,\Sigma)\simeq L_\alpha^\pf\oplus L_\beta^\pf\oplus L_\gamma^\pf$ is the only non-trivial space in $H^\varphi(Y,\Sigma)$. 
 Moreover, the isomorphism of chain complexes
 $C^\varphi(Y,\Sigma) \simeq C^\varphi(H_\alpha,\Sigma) \oplus C^\varphi(H_\beta,\Sigma) \oplus C^\varphi(H_\gamma,\Sigma)$ provides 
 $$\tau^\varphi(Y,\Sigma;\hYS)=\tau^\varphi(H_\alpha,\Sigma;\ha)\tau^\varphi(H_\beta,\Sigma;\hb)\tau^\varphi(H_\gamma,\Sigma;\hc).$$
 Fix $\nu\in\{\alpha,\beta,\gamma\}$. The handlebody $H_\nu$ is obtained from $\Sigma$ by adding meridian disks $D_i^\nu$ such that $\partial D_i^\nu=\nu_i$  and the $3$-cell $H_\nu$. The associated $\F$-complex of $(H_\nu,\Sigma;\varphi)$ is $C^\varphi(H_\nu,\Sigma)$:
 $$ \ 0\to C_3^\varphi(H_\nu,\Sigma)\to C_2^\varphi(H_\nu,\Sigma)\to 0$$ 
 with bases $H_\nu$ and $(D_i^\nu)_i$. Choose the lifts so that
 $$\partial H_\nu=\sum_{1\leq i \leq g}(\varphi(\nu_i^*)-1)\, D_i^\nu\ \in C_2^\varphi(H_\nu,\Sigma).$$
 Via the identification $H_2^\varphi(H_\nu,\Sigma)\simeq L_\nu^\pf$, the basis $(\frac{1}{\varphi(\nu_1^*)-1}\, D^\nu_2, D^\nu_3,\dots, D^\nu_g)$ coincides with $h_\nu$. Hence
 $\tau^\varphi(H_\nu,\Sigma;h_\nu)=\left[\frac{\partial H_\nu\cdot h_\nu}{( D_i^\nu)_{1\leq i\leq g}}\right]=1$. 
\end{proof}

\begin{lemma} \label{sympl}
Let $j \in \{1,2,3 \}$ and $\nu,\nu' \in\{\alpha,\beta,\gamma\}$ be such that $L_\nu^\pf\cap L_{\nu'}^\pf\simeq\partial X_j$.
There is a symplectic basis $(\xi_i,\xi_i^*)_{1\leq i\leq g}$ of $\Sigma$ such that the family $h_{\nu\nu'}=(\frac{1}{\varphi(\xi_1^*)-1}\xi_2,\xi_3,\dots,\xi_k)$ is a basis of $L_\nu^\pf\cap L_{\nu'}^\pf$.
\end{lemma}
\begin{proof}
The surface $\Sigma$ together with the families of curves $\nu$ and $\nu'$ is a Heegaard diagram of $\partial X_j\simeq\sharp^k(S^1\times S^2)$.
Hence there is a symplectic basis $(\xi_i,\xi_i^*)_{1\leq i\leq g}$ such that performing handleslides changes $(\nu_i)_{1\leq i\leq g}$ into $(\xi_i)_{1\leq i\leq g}$ and $(\nu_i')_{1\leq i\leq g}$ into $(\xi_1,\dots,\xi_k,\xi_{k+1}^*,\dots,\xi_g^*)$. Permuting the indices if necessary, we assume $\varphi(\xi_1^*)\neq1$. 
 Now $(\xi_1,\dots,\xi_g)$ is a system of meridians for $H_\nu$ and $(\xi_1,\dots,\xi_k,\xi_{k+1}^*,\dots,\xi_g^*)$ for $H_{\nu'}$. By Lemma \ref{lemmaBaseLag} the families $(\frac{1}{\varphi(\xi_1^*)-1}\xi_2,\xi_3,\dots,\xi_g)$ and $(\frac{1}{\varphi(\xi_1^*)-1}\xi_2,\xi_3,\dots,\xi_k,\xi_{k+1}^*,\dots,\xi_g^*)$ are bases of $L_\nu^\pf$ and $L_{\nu'}^\pf$ respectively.
\end{proof}

\begin{lemma} \label{lemmatorsionXY}
 Via the identification $H_3^\varphi(X,Y)\simeq (L_\alpha^\pf\cap L_\beta^\pf)\oplus (L_\beta^\pf\cap L_\gamma^\pf) \oplus (L_\gamma^\pf\cap L_\alpha^\pf)$, the family $\hXY=\hab.\hbc.\hca$ is a homology basis of the pair $(X,Y;\varphi)$ and we have \mbox{$\tau^\varphi(X,Y;\hXY)=1$.}
\end{lemma}
\begin{proof}
 Fix $j \in \{1,2,3 \}$ and $\nu,\nu' \in\{\alpha,\beta,\gamma\}$ such that $\partial X_j\simeq H_\nu\cup H_{\nu'}$. Let $(\xi_i,\xi_i^*)_{1\leq i\leq g}$ be a symplectic basis of $\Sigma$ given by Lemma \ref{sympl}. For $i=1,\dots,g$, let $D_i$ be a meridian disk of $H_\nu$ with $\partial D_i=\xi_i$. Similarly let $\Delta_i$ be meridians disks of $H_{\nu'}$ such that $\partial \Delta_i=\xi_i$ for $1\leq i\leq k$ and $\partial \Delta_i=\xi_i^*$ for $k<i\leq g$. The handlebody $X_j$ is obtained from $\partial X_j$ by adding 3--cells $B_i$ such that $\partial B_i=D_i-\Delta_i$ for $1\leq i\leq k$ and a 4--cell $X_j$. The associated $\F$-complex is $C^\varphi(X_j,\partial X_j)$:
 $$ \ 0 \rightarrow C_4^\varphi(X_j,\partial X_j) \rightarrow C_3^\varphi(X_j,\partial X_j) \rightarrow 0$$
 with bases $X_j$ and $(B_i)_{1\leq i\leq k}$. Choose the lifts so that 
 $$\partial  X_j=\sum_{1\leq i \leq k}(\varphi(\xi_i^*)-1)\, B_i\ \in C_3^\varphi(X_j,\partial X_j).$$ Hence $H^\varphi_\ell(X_j,\partial X_j)=0$ if $\ell\neq3$. % and $(\frac{1}{\varphi(\xi_1^*)-1}B_2,B_3,\dots,B_k)$ is a basis of $H_3(X_i,\partial X_i)$.
 Moreover, by Blanchfield duality, $H_3(X_j,\partial X_j)\simeq H_1^\varphi(X_j)$. Since $X_j$ is obtained from $\partial X_j$ by adding $3$ and $4$-cells, $H_1^\varphi(X_j) \simeq H_1^\varphi(\partial X_j)$. Finally, by Lemma \ref{lemmaXYhomtwist} and Blanchfield duality, $H_1^\varphi(\partial X_j) \simeq H_2^\varphi(\partial X_j) \simeq H_3^\varphi(X_j,\partial X_j) \simeq L_\nu^\pf\cap L_{\nu'}^\pf$.
 A basis of $H_2^\varphi(\partial X_j)$ is obtained by identifying $\xi_i$ with $D_i\cup\Delta_i$. Its Blanchfield dual basis of $H_1^\varphi(\partial X_j)$ is $(\frac{1}{\varphi(\xi_1^*)-1}\xi_2^*,\xi_3^*,\dots,\xi_k^*)$, which is also a basis of $H_1^\varphi(X_j)$. Its Blanchfield dual basis of $H_3(X_j,\partial X_j)$ is $(\frac{1}{\varphi(\xi_1^*)-1} B_2, B_3,\dots, B_k)$ and $$\tau^\varphi(X_j,\partial X_j;h_{\nu\nu'})=\left[\frac{\partial  X_j.h_{\nu\nu'}}{(B_i)_{1\leq i\leq k}}\right]=1.$$ 
 Conclude with $C^\varphi(X,Y)\simeq\oplus_{j=1}^3 C^\varphi(X_j,\partial X_j)$.
\end{proof}

\subsection{Computation of the torsion of $(X;\varphi)$}

In this subsection, we prove Theorem~\ref{thtorsion}. 

Let $\hX$ be a homology basis of $(X;\varphi)$. Let $\hY$ be a homology basis of $Y$, with $\hY^1=\hX^1$ ---recall there is a natural identification $H_1^\varphi(X)\simeq H_1^\varphi(Y)$. Let $\hS$ be a homology basis of $\Sigma$ as provided by Lemma \ref{lemmaBaseSigma}. For the pair $(Y,\Sigma)$, fix the homology basis $\hYS=\ha.\hb.\hc$. For the pair $(X,Y)$, fix the homology basis $\hXY=\hab.\hbc.\hca$. 

\begin{lemma} \label{lemmatorsionY}
 The exact sequence in homology associated with the pair $(Y,\Sigma;\varphi)$ reduces to 
 \begin{equation} \tag{$\HH_Y$} \label{HN}
  0\longrightarrow H_2^\varphi(Y)\longrightarrow H_2^\varphi(Y,\Sigma) \longrightarrow H_1^\varphi(\Sigma)\longrightarrow H_1^\varphi(Y)
  \longrightarrow 0 
 \end{equation}
 and we have $\tau^\varphi(Y; \hY)=\tau(\HH_Y).$
\end{lemma}
\begin{proof}
 The exact sequence of complexes associated with the pair $(Y,\Sigma)$ provides the following equality (see Lemma \ref{lem:torsion_as_function}):
 $$\tau^\varphi(Y; \hY) = \tau^\varphi(\Sigma;\hS) \, \tau^\varphi( Y,\Sigma;\hYS) \, \tau(\HH_Y).$$ 
 The result follows from Lemmas \ref{lemmaBaseSigma} and \ref{lemmaBaseYSigma}.
\end{proof}

\begin{lemma} \label{lemmatorsionHX}
 Let $\HH_{X}$ be the exact sequence in homology associated with the pair $(X,Y)$. We have $\tau(\HH_{X}) = \tau(\HH_X'),$ where $\HH_X'$ is the following part of the sequence $\HH_{X}$ itself:
 \begin{equation} \tag{$\HH_X'$} \label{H2}
  0\to H_3^\varphi(X)\to H_3^\varphi(X,Y)\to H_2^\varphi(Y)\to H_2^\varphi(X)\to 0.
 \end{equation}
\end{lemma}
\begin{proof}
 The sequence $\HH_X$ is composed of $\HH_X'$ and $0\to H^\varphi_1(Y)\to H^\varphi_1(X)\to0$. 
\end{proof}

\begin{proof}[Proof of Theorem \ref{thtorsion}]
Here, $\hX$ is the homology basis $h$ of the statement. The exact sequence with coefficients in $\F$ associated with the pair $(X,Y)$ induces the following equality:
$$\tau^\varphi(X;\hX)=  \tau^\varphi(Y;\hY) \,\tau^\varphi(X,Y;\hXY) \,\tau(\HH_{X}).$$
Hence Lemmas~\ref{lemmatorsionXY} and~\ref{lemmatorsionY} give:
$$\tau^\varphi(X;\hX)=\tau(\HH_Y)\,\tau(\HH_X').$$
By Lemmas \ref{lemmaYShomtwist} and~\ref{lemmatorsionY}, the sequence \ref{HN} writes:
$$0\to H_2^\varphi(Y) \to L_\alpha^\pf\oplus L_\beta^\pf\oplus L_\gamma^\pf \to H_1^\varphi(\Sigma) \to H_1^\varphi(Y)\to 0.$$
Fixing a basis $s$ for $L_\alpha^\pf+L_\beta^\pf+L_\gamma^\pf\subset H_1^\varphi(\Sigma)$, we get
$$\tau(\HH_Y)=\left[\frac{\hY^2.s}{\ha.\hb.\hc}\right]^{-1}\left[\frac{s.\hX^1}{\hS}\right].$$
By Lemmas \ref{lemmaXYhomtwist} and~\ref{lemmatorsionHX}, the sequence \ref{H2} writes:
$$0\to H_3^\varphi(X) \to (L_\alpha^\pf\cap L_\beta^\pf)\oplus
(L_\beta^\pf\cap L_\gamma^\pf)\oplus(L_\gamma^\pf\cap L_\alpha^\pf) \fl{\zeta} H_2^\varphi(Y)\to H_2^\varphi(X) \to 0.$$
Fixing a basis $t$ for $\mathrm{Im}(\zeta)\subset H_2^\varphi(Y)$, we get
$$\tau(\HH_X')=\left[\frac{\hX^3.t}{\hab.\hbc.\hca}\right]\left[\frac{t.\hX^2}{\hY^2}\right]^{-1}.$$
Fixing the complex basis $c=(\hab.\hbc.\hca,\ha.\hb.\hc,\hS)$ for the complex $C^\pf$, we have:
$$\tau(C^\pf;c,\hX)=\left[\frac{\hX^3.t}{\hab.\hbc.\hca}\right]\left[\frac{t.\hX^2.s}{\ha.\hb.\hc}\right]^{-1}\left[\frac{s.\hX^1}{\hS}\right],$$
so that we get the desired equality.
\end{proof}

\section{Computation of $\tau^\varphi(X)$ via $(\p X,\star)$} \label{secTorsionXstar}

In this section, we compute the torsion of $(\p X,\star;\varphi)$ and relate it to the torsion of $(X;\varphi)$. For $\nu\in\{\alpha,\beta,\gamma\}$, let $\p L_\nu^\pf$ be the subspace of $H_1^\varphi(\p\Sigma,\star;\F)$ generated by the homology classes of the curves $\nu_i$. 

\begin{lemma} \label{lemmaTorsXpX}
 The $\Lambda$--modules $H_1^\varphi(\p\Sigma,\star;\Lambda)$, $\p L^\pl_\nu$ and $\p L^\pl_\nu\cap \p L^\pl_{\nu'}$ are free $\Lambda$--modules of rank $2g$, $g$ and $k$ respectively. 
 Moreover, $H_1^\varphi(\p\Sigma,\star;\F)=H_1^\varphi(\p\Sigma,\star;\Lambda)\otimes_\Lambda\F$, $\p L^\pf_\nu=\p L^\pl_\nu\otimes_\Lambda\F$ and $\p L^\pf_\nu\cap \p L^\pf_{\nu'}=(\p L^\pl_\nu\cap \p L^\pl_{\nu'})\otimes_\Lambda\F$.
\end{lemma}

\begin{proof}
 Set $R=\Lambda$ or $\F$. 
 Fix $\nu\in\{\alpha,\beta,\gamma\}$. Let $(\nu_i^*)_{1\leq i\leq g}$ be a family of curves on $\p\Sigma$ such that 
 $(\nu_i,\nu_i^*)_{1\leq i\leq g}$ is a symplectic basis for $\Sigma$. Assume the removed ball $B$ is such that $\partial D=\prod_{i=1}^g[\nu_i,\nu_i^*]$. Consider a CW--decomposition of $\p\Sigma$ with one $0$-cell $\star$, one $2$-cell $\p \Sigma$ and $1$-cells $\nu_i$, $\nu_i^*$ and $\partial D$.
 The associated $\F$-complex is $C^\varphi(\p\Sigma,\star)$:
 $$  0\to C_2^\varphi(\p\Sigma,\star)\to C_1^\varphi(\p\Sigma,\star)\to 0.$$ 
 Choose a lift of $\p\Sigma$ such that:
 $$\partial \p \Sigma=-\partial  D+\sum_{1\leq i \leq g}(\varphi(\nu_i^*)-1)\, \nu_i\ \in C_1^\varphi(\p\Sigma,\star).$$
 The only non-trivial homology $R$--module of $(\p\Sigma,\star)$ is $H_1^\varphi(\p\Sigma,\star;R)\simeq R^{2g}$ generated by $\nu_i$ and $\nu_i^*$ for $i=1,\dots,g$. Since $\p L_\nu^\pr$ is the submodule of $H_1 ^\varphi(\p\Sigma,\star;R)$ generated by the $\nu_i$, we get $\p L_\nu^\pr\simeq R^g$. 
 
 A similar computation can be done for $(\p\Sigma,\star)$ from any symplectic basis for $\Sigma$. As in Lemma \ref{sympl}, if $\nu$ and $\nu'$ are distinct, there exists a symplectic basis $(\xi_i,\xi_i^*)_{1\leq i\leq g}$ for $\Sigma$ such that $\p L_\nu^\pr\cap\p L_{\nu'}^\pr\simeq R^k$ freely generated by $(\xi_i)_{1\leq i\leq k}$.
\end{proof}

If $\Gamma$ is a free $\Lambda$--module, a {\em $\Lambda$--basis} of $\Gamma\otimes_\Lambda\F$ is a basis $b\otimes1$ where $b$ is a basis of $\Gamma$.

\begin{theorem} \label{thtorsionXstar}
 The twisted homology of $(\p X,\star)$ canonically identifies with the homology of the following $\F$--complex $\p C^\pf$:
 $$0 \to (\p L_\alpha^\pf \cap \p L_\beta^\pf)\oplus(\p L_\beta^\pf \cap \p L_\gamma^\pf)\oplus(\p L_\gamma^\pf \cap \p L_\alpha^\pf)
 \fl{\zeta} \p L_\alpha^\pf \oplus \p L_\beta^\pf \oplus \p L_\gamma^\pf \fl{\iota} H_1^\varphi(\p\Sigma,\star;\F) \to 0,$$
 where $\zeta(x,y,z)=(x-z,y-x,z-y)$ and $\iota$ is defined by the inclusions $\p L_\nu^\pf\hookrightarrow H_1^\varphi(\p\Sigma,\star;\F)$.
 Moreover, for any complex $\Lambda$--basis $\p c$ for $\p C^\pf$ and any homology basis $\p h$ for $(\p X,\star)$ and $\p C^\pf$, we have: 
 $$\tau^\varphi(\p X,\star;\p h)=\tau(\p C^\pf;\p c,\p h) \quad\textrm{in }\F/\Lambda^*.$$
\end{theorem}

\begin{proof}
 Using Lemma \ref{lemmaTorsXpX}, the whole sections \ref{sechomologyXtwist} and \ref{secTorsion} adapt to the setting of $(\p X,\star)$, providing the result. The independance with respect to the choice of a $\Lambda$--bases over $\F$ for $\p C^\pf$ is due to the fact that a change of such bases modifies the torsion by an element of $\Lambda^*$.
\end{proof}

\begin{lemma} \label{lemmaHomXpX}
For all $\ell \neq 1,2,3$, one has $H_\ell^\varphi(X;\F) = H_\ell^\varphi(X,\star)=H_\ell^\varphi(\p X,\star;\F)=0$.
 Moreover, there are natural identifications  of $\F$--vector spaces
 $$H_1^\varphi(X,\star)\simeq H_1^\varphi(\p X,\star), \ H_2^\varphi(X)\simeq H_2^\varphi(X,\star)\simeq H_2^\varphi(\p X,\star) \text{ and }H_3^\varphi(X)\simeq H_3^\varphi(X,\star)$$ and short exact sequences  of $\F$--vector spaces
 $$ 0 \rightarrow H_1^\varphi(X) \rightarrow H_1^\varphi(\p X,\star)\to H_0^\varphi(\star) \rightarrow 0,% \quad \textrm{and}
  \ 0 \rightarrow H_4^\varphi(B,\partial B) \rightarrow H_3^\varphi(\p X,\star) \rightarrow H_3^\varphi(X) \rightarrow 0.$$
\end{lemma}
\begin{proof}
 The result follows from the exact sequence in homology of the pair $(X,\star)$ and from the exact sequence in homology of the triple $(X,\p X,\star)$ combined with the excision equivalence $(X,\p X)\simeq(B,\partial B)$.
\end{proof}

\begin{proposition} \label{propTorsXpX}
 Let $h$ be a homology basis of $X$. Let $u\in H_1(\p X,\star)$ satisfy $\varphi(u)\neq1$. Then $\p h=(\partial B.h^3,h^2,h^1.u)$ is a homology basis of $(\p X,\star)$ and $$(\varphi(u)-1)\,\tau^\varphi(X;h)=\tau^\varphi(\p X,\star;\p h).$$
\end{proposition}

\begin{proof}
 Thanks to Lemma \ref{lemmaHomXpX}, $\p h$ is a homology basis for $(\p X,\star)$ and $\hat h=(h^3,h^2,h^1.u)$ is a homology basis for $(X,\star)$. 
 The short exact sequences of complexes 
 $0 \to C^\varphi(\star)\to C^\varphi(X)\to C^\varphi(X,\star)\to 0$ and $0\to C^\varphi(\p X,\star)\to C^\varphi(X,\star)\to C^\varphi(B,\partial B)\to 0$
 provide 
 $$\tau^\varphi(X;h)=\tau^\varphi(\star;\star) \, \tau^\varphi(X,\star;\bar h) \, \tau(\Seq_1)$$
  and  
$$ \tau^\varphi(X,\star;\bar h)=\tau^\varphi(\p X,\star;\p h) \, \tau^\varphi(B,\partial B;B) \, \tau(\Seq_2)$$ 
 where $\Seq_1$ and $\Seq_2$ are the associated exact sequences in homology.
 One easily checks that $\tau^\varphi(\star;\star)=1$ and $\tau^\varphi(B,\partial B;B)=1$, so that:
 $$\tau^\varphi(X;h)=\tau^\varphi(\p X,\star;\p h) \, \tau(\Seq_1) \, \tau(\Seq_2).$$
 A straightforward computation shows that $\tau(\Seq_1)=(\varphi(u)-1)^{-1}$ and $\tau(\Seq_2)=1$.
\end{proof}

\section{Intersection forms} \label{secintform}

In this section, we prove the results on intersection forms. Along the proofs, we give interpretations of the modules of the complexes $C$ and $C^\pf$ as modules of chains. 

We first reprove the expression of the intersection form on $H_2(X;\Z)$ given in \cite{FKSZ18} with our approach. Following Wall \cite{Wa69}, define the symmetric form 
$$\lambda: \frac{L_\alpha \cap (L_\beta + L_\gamma)}{(L_\alpha\cap L_\beta)+ (L_\alpha\cap L_\gamma)} 
 \times \frac{L_\alpha \cap (L_\beta + L_\gamma)}{ (L_\alpha \cap L_\beta)+ (L_\alpha \cap L_\gamma)} \longrightarrow \Z$$
as follows. For $a,a'  \in L_\alpha \cap (L_\beta + L_\gamma)$ and $b\in L_\beta$, $c\in L_\gamma$ such that $a+b+c=0$, 
set $$ \lambda(a,a'):= \langle c, a' \rangle_{\Sigma}.$$ 

\begin{proposition}[\cite{FKSZ18}] \label{propintform} 
Let $\langle \cdot,\cdot \rangle_{X}$ be the intersection form of $X$.
There is an isomorphism
$$ \left(H_2(X;\Z); \langle \cdot,\cdot \rangle_{X} \right) \simeq \left(\frac{L_\alpha \cap 
(L_\beta + L_\gamma)}{ (L_\alpha \cap L_\beta)+ (L_\alpha \cap L_\gamma)}; \lambda \right).$$
\end{proposition}
\begin{proof}
 By Theorem \ref{thhomologyXnotwist}, 
 $$H_2(X)\simeq\frac{\ker\left(L_\alpha\oplus L_\beta\oplus L_\gamma\to H_1(\Sigma;\Z)\right)}
 {(L_\alpha\cap L_\beta)\oplus(L_\beta\cap L_\gamma)\oplus(L_\gamma\cap L_\alpha)}\simeq
 \frac{L_\alpha\cap(L_\beta+ L_\gamma)}{(L_\alpha\cap L_\beta)+(L_\alpha\cap L_\gamma)}.$$
 Following the trisection, $X$ is built from $D^2\times\Sigma$ by adding $2,3,4$--cells attached to $S^1\times\Sigma$, see Section \ref{subsecReconstruction}.  
 Let $p_0, p_\alpha, p_\beta, p_\gamma$ be the points in $D^2$ defined as in Figure~\ref{figreconstructX}.
 Given $\mu=(\mu_\alpha,\mu_\beta,\mu_\gamma)$ in $\ker(L_\alpha\oplus L_\beta\oplus L_\gamma\to H_1(\Sigma))$ and a point $q\neq p_0 \in \Int{D^2}$, we construct a $2$-cycle
 $S_q(\mu) \in C_2(X)$ as follows. For all $\nu \in \{ \alpha, \beta, \gamma \}$, write $\mu_\nu$ as a linear combination of the $\nu_i$ and define $D_\nu(\mu)$ as the corresponding disjoint union of meridian disks bounded by parallel copies of the $\{p_\nu\}\times\nu_i$. Then define $$S_q(\mu):=S_\alpha+S_\beta+S_\gamma+T(\mu),$$ where:
 \begin{itemize}
  \item $S_\nu=D_\nu(\mu) \cup \left([p_\nu,q]\times\mu_\nu \right)$,
  \item $T(\mu)$ is a 2--chain with support contained in $\{q\}\times\Sigma$ with $\partial T(\mu)=\mu_\alpha+\mu_\beta+\mu_\gamma$.
 \end{itemize}
 Let now $\mu=(\mu_\alpha,\mu_\beta,\mu_\gamma)$ and $\mu'=(\mu_\alpha',\mu_\beta',\mu_\gamma')$ be in $\ker(L_\alpha\oplus L_\beta\oplus L_\gamma\to H_1(\Sigma))$. 
 Fix $q$ and $q'$ in $\Int{D^2}$ such that $p_0=(p_\gamma,q)\cap(p_\alpha,q')$. The $2$-cycles $S_q(\mu)$ and $S_{q'}(\mu')$ intersect transversally in $\{ p_0 \} \times \Sigma$ and
 $$ \langle S_q(\mu),S_{q'}(\mu')\rangle_X=\langle\mu_\gamma,\mu_\alpha'\rangle_\Sigma = \lambda(\mu,\mu'),$$
 where we assume that $D^2$ is oriented by the oriented basis $(\vec{p_0p_\alpha},\vec{p_0p_\gamma})$.
\end{proof}

\begin{figure} [htb]
 \begin{center}
 \begin{tikzpicture} [scale=0.9]
 \begin{scope}
  \draw (0,0) circle (2);
  \draw (2,0) node {$\scriptstyle{\bullet}$} node[right] {$p_\gamma$} (-0.5,0) node {$\scriptstyle{\bullet}$} node[left] {$q$};
  \draw[rotate=120] (2,0) node {$\scriptstyle{\bullet}$} node[above] {$p_\alpha$} (1,0);
  \draw[rotate=240] (2,0) node {$\scriptstyle{\bullet}$} node[below] {$p_\beta$} (1,0);
  \draw (0,0) node {$\scriptstyle{\bullet}$} node[above right] {$p_{\scriptscriptstyle{0}}$};
  \draw (-0.5,0) -- (2,0) (-1,1.73) -- (-0.5,0) -- (-1,-1.73);
  \node (Da) at (-3,1.7) {$D_\alpha(\mu)$}; \node (pa) at (-1,1.7) {}; \draw[->] (Da) -- (pa);
  \node (Db) at (-3,-1.7) {$D_\beta(\mu)$}; \node (pb) at (-1,-1.7) {}; \draw[->] (Db) -- (pb);
  \node (Dg) at (2.6,1.2) {$D_\gamma(\mu)$}; \node (pg) at (2,0) {}; \draw[->] (Dg) -- (pg);
  \node (mua) at (-3,0.7) {$\mu_\alpha$}; \node (sa) at (-0.75,0.87) {}; \draw[->] (mua) -- (sa);
  \node (mub) at (-3,-0.7) {$\mu_\beta$}; \node (sb) at (-0.75,-0.87) {}; \draw[->] (mub) -- (sb);
  \node (mug) at (2,2) {$\mu_\gamma$}; \node (sg) at (0.8,0) {}; \draw[->] (mug) -- (sg);
  \node[align=center] (T) at (2.2,-2) {$T(\mu)$}; \node (q1) at (-0.5,0) {}; \draw[->] (T) -- (q1);
 \end{scope}
 \begin{scope} [xshift=8.8cm]
  \draw (0,0) circle (2);
  \draw (2,0) node {$\scriptstyle{\bullet}$} node[right] {$p_\gamma$};
  \draw[rotate=120] (2,0) node {$\scriptstyle{\bullet}$} node[above] {$p_\alpha$} (-0.5,0) node {$\scriptstyle{\bullet}$} node[below right] {$q'$};
  \draw[rotate=240] (2,0) node {$\scriptstyle{\bullet}$} node[below] {$p_\beta$};
  \draw (0,0) node {$\scriptstyle{\bullet}$} node[left] {$p_{\scriptscriptstyle{0}}$};
  \draw[rotate=120] (-0.5,0) -- (2,0) (-1,1.73) -- (-0.5,0) -- (-1,-1.73);
  \node (Da') at (-3,1.7) {$D_\alpha(\mu')$}; \node (pa') at (-1,1.7) {}; \draw[->] (Da') -- (pa');
  \node (Db') at (-3,-1.7) {$D_\beta(\mu')$}; \node (pb') at (-1,-1.7) {}; \draw[->] (Db') -- (pb');
  \node (Dg') at (2.6,1.2) {$D_\gamma(\mu')$}; \node (pg') at (2,0) {}; \draw[->] (Dg') -- (pg');
  \node (mua') at (-3,0.7) {$\mu_\alpha'$}; \node (sa') at (-0.4,0.69) {}; \draw[->] (mua') -- (sa'); 
  \node (mub') at (0.3,-2.5) {$\mu_\beta'$}; \node (sb') at (-0.4,-1) {}; \draw[->] (mub') -- (sb');
  \node (mug') at (2.5,-1) {$\mu_\gamma'$}; \node (sg') at (1.1,-0.2) {}; \draw[->] (mug') -- (sg');
  \node[align=center] (T) at (1.2,2.4) {$T(\mu')$}; \node (q2) at (0.25,-0.4) {}; 
  \draw[->] (T) -- (q2);
 \end{scope}
 \end{tikzpicture}
 \end{center} \caption{The 2--cycles $S$ and $S'$.} \label{figXsurfaces}
\end{figure}

We now prove Theorem \ref{thtwistedintform}, which is the analogue in the twisted setting of Proposition~\ref{propintform}. 
\begin{proof}[Proof of Theorem \ref{thtwistedintform}]
 All the curves of the families $\alpha$, $\beta$, $\gamma$ have their homology classes in $\ker(\varphi)$, so that they lift as loops. 
 Moreover, the meridian disks and the paths of Figure~\ref{figXsurfaces} drawn on $D^2$ are contractible, thus they also lift as disks and paths. 
 Hence the result follows from the very same argument as in Proposition \ref{propintform}.
\end{proof}

We turn to the intersection form on $H_1(X;\Z)\times H_3(X;\Z)$.

\begin{proposition} \label{propintfromH13}
 There is an isomorphism
 $$ \left(H_1(X;\Z)\times H_3(X;\Z); \langle \cdot,\cdot \rangle_{X} \right) \cong \left(\frac{H_1(\Sigma)}{L_\alpha + L_\beta + L_\gamma}\times(L_\alpha\cap L_\beta\cap L_\gamma); \langle\cdot,\cdot\rangle_\Sigma \right).$$
\end{proposition}

 \begin{figure} [htb]
 \begin{center}
 \begin{tikzpicture}
 \begin{scope}
  \foreach \t in {0,120,240} {
  \draw[rotate=\t] (0.5,0) -- (2,0);}
  \draw (0,0) circle (2);
  \draw[pattern=north east lines] (0.5,0) -- (-0.25,0.43) -- (-0.25,-0.43) -- (0.5,0);
  \draw (2,0) node {$\scriptstyle{\bullet}$} node[right] {$p_\gamma$} (0.5,0) node {$\scriptstyle{\bullet}$} node[below] {$q_\gamma$};
  \draw[rotate=120] (2,0) node {$\scriptstyle{\bullet}$} node[above] {$p_\alpha$} (0.5,0) node {$\scriptstyle{\bullet}$} node[left] {$q_\alpha$};
  \draw[rotate=240] (2,0) node {$\scriptstyle{\bullet}$} node[below] {$p_\beta$} (0.5,0) node {$\scriptstyle{\bullet}$} node[left] {$q_\beta$};
  \node (Da) at (-3,1.7) {$F_\alpha$}; \node (pa) at (-0.8,1.5) {}; \draw[->] (Da) -- (pa);
  \node (Db) at (-3,-1.7) {$F_\beta$}; \node (pb) at (-0.8,-1.5) {}; \draw[->] (Db) -- (pb);
  \node (Dg) at (2.6,1.2) {$F_\gamma$}; \node (pg) at (1.7,-0.05) {}; \draw[->] (Dg) -- (pg);
  \node (mu) at (0,3) {$\mu$}; \node (s) at (0,0) {}; \draw[->] (mu) -- (s);
  \node (M1) at (-2.8,0) {$M_1$}; \node (s1) at (-1,0) {}; \draw[->] (M1) -- (s1);
  \node (M2) at (1.2,-2.4) {$M_2$}; \node (s2) at (0.6,-0.8) {}; \draw[->] (M2) -- (s2);
  \node (M3) at (1.2,2.4) {$M_3$}; \node (s3) at (0.6,0.8) {}; \draw[->] (M3) -- (s3);
 \end{scope}
 \end{tikzpicture}
 \end{center} \caption{The 3--cycle $M$ associated to $\mu$.} \label{fig3cycle}
 \end{figure}

\begin{proof}[Proof of Proposition \ref{propintfromH13}]
 Let $\mu\in L_\alpha\cap L_\beta\cap L_\gamma$. By a slight abuse of notation, we use the same letter $\mu$ for a representative of it on $\Sigma$. We will construct a 3--cycle $M$ associated with $\mu$ that intersects $\Sigma$ along $\mu$. Once again, we view $X$ as reconstructed from the trisection diagram. 
 Let $q_\alpha,q_\beta,q_\gamma$ be the points on $D^2$ represented in Figure \ref{fig3cycle} and let $V$ be the hatched triangle they define. We will complete the 3--chain $M_0=V\times\mu$ into a 3--cycle. 
 For each $\nu\in\{\alpha,\beta,\gamma\}$, $\mu$ bounds a surface $F_\nu$ properly embedded in $H_\nu$. For $i,\nu,\nu'$ such that $\partial X_i=H_\nu\cup H_{\nu'}$, since $X_i\cong \big(X_i\setminus(\Int V\times\Sigma)\big)$ has trivial second homology, the closed surface $F_\nu\cup([q_\nu,q_{\nu'}]\times\mu)\cup F_{\nu'}$ bounds a $3$--cycle $M_i\subset \big(X_i\setminus(\Int V\times\Sigma)\big)$. 
 Finally, $M=\sum_{0\leq i\leq 3}M_i$ is a 3--cycle associated with $\mu$. Then, for any $\mu'\in H_1(\Sigma)$, we have 
 $\langle \mu',M\rangle_X=\langle \mu',\mu_\alpha\rangle_\Sigma=\langle\mu',\mu\rangle_\Sigma.$
\end{proof}

A similar proof yields Proposition \ref{propintfromH13twisted}.

\section{Examples} \label{secex}

\subsection*{Example 1}
\begin{figure}[htb]
\begin{center}
\begin{tikzpicture} 
\begin{scope}[scale=0.6]
%tour
\draw (0,0) ..controls +(0,1) and +(-2,1) .. (4,2);
\draw (4,2) ..controls +(2,-1) and +(-2,-1) .. (8,2);
\draw (8,2) ..controls +(2,1) and +(-1.2,0) .. (12,1.3);
\draw (0,0) ..controls +(0,-1) and +(-2,-1) .. (4,-2);
\draw (4,-2) ..controls +(2,1) and +(-2,1) .. (8,-2);
\draw (8,-2) ..controls +(2,-1) and +(-1.2,0) .. (12,-1.3);
\draw (14,2) ..controls +(2,1) and +(0,1) .. (18,0);
\draw (14,-2) ..controls +(2,-1) and +(0,-1) .. (18,0);
\draw (12,1.3) ..controls +(0.5,0) and +(-1,-0.5) .. (14,2);
\draw (12,-1.3) ..controls +(0.5,0) and +(-1,0.5) .. (14,-2);
%trous
\newcommand{\trou}[2]{\draw[xshift=#1cm,yshift=#2cm]
(-1,0.1) ..controls +(0.5,-0.25) and +(-0.5,-0.25) .. (1,0.1)
(-0.7,0) ..controls +(0.6,0.2) and +(-0.6,0.2) .. (0.7,0)}
\trou{3}{0};
\trou{9}{0};
\trou{15}{0};
%pt
\draw (2.95,-0.8) .. controls +(1,0.2) and +(-1,0.2) .. (9.1,-0.8) node {$\star$} .. controls +(1,0.2) and +(-1,0.2) .. (15.2,-0.8);
\end{scope}
\begin{scope} [xscale=0.6,yshift=0.15cm]
%gamma
\draw[green] (3.2,-0.2) ..controls +(0.2,-0.5) and +(0.2,0.5) .. (3.2,-1.5);
\draw[dashed,green] (3.2,-0.2) ..controls +(-0.2,-0.5) and +(-0.2,0.5) .. (3.2,-1.5);
\draw [->,green] (3.35,-0.8)--(3.35,-0.9);
\draw[green] (3.2,-1.5) node[below right] {$\scriptstyle{\gamma_1}$};
%beta
\draw[blue] (3,-0.2) ..controls +(0.2,-0.5) and +(0.2,0.5) .. (3,-1.5);
\draw[dashed,blue] (3,-0.2) ..controls +(-0.2,-0.5) and +(-0.2,0.5) .. (3,-1.5);
\draw[->,blue] (3.15,-0.8)--(3.15,-0.9);
\draw[blue] (3,-1.5) node[below] {$\scriptstyle{\beta_1}$};
%alpha
\draw[red] (2.8,-0.2) ..controls +(0.2,-0.5) and +(0.2,0.5) .. (2.8,-1.5);
\draw[dashed,red] (2.8,-0.2) ..controls +(-0.2,-0.5) and +(-0.2,0.5) .. (2.8,-1.5);
\draw [->,red] (2.95,-0.8)--(2.95,-0.9);
\draw[red] (2.8,-1.5) node[below left] {$\scriptstyle{\alpha_1}$};
\end{scope}
\begin{scope} [xscale=0.6,xshift=6cm,yshift=0.15cm]
%gamma
\draw[green] (3.2,-0.2) ..controls +(0.2,-0.5) and +(0.2,0.5) .. (3.2,-1.5);
\draw[dashed,green] (3.2,-0.2) ..controls +(-0.2,-0.5) and +(-0.2,0.5) .. (3.2,-1.5);
\draw [->,green] (3.35,-0.8)--(3.35,-0.9);
\draw[green] (3.2,-1.5) node[below right] {$\scriptstyle{\gamma_2}$};
%beta
\draw[blue] (3,-0.2) ..controls +(0.2,-0.5) and +(0.2,0.5) .. (3,-1.5);
\draw[dashed,blue] (3,-0.2) ..controls +(-0.2,-0.5) and +(-0.2,0.5) .. (3,-1.5);
\draw[->,blue] (3.15,-0.8)--(3.15,-0.9);
\draw[blue] (3,-1.5) node[below] {$\scriptstyle{\beta_2}$};
%alpha
\draw[red] (2.8,-0.2) ..controls +(0.2,-0.5) and +(0.2,0.5) .. (2.8,-1.5);
\draw[dashed,red] (2.8,-0.2) ..controls +(-0.2,-0.5) and +(-0.2,0.5) .. (2.8,-1.5);
\draw [->,red] (2.95,-0.8)--(2.95,-0.9);
\draw[red] (2.8,-1.5) node[below left] {$\scriptstyle{\alpha_2}$};
\end{scope}
\begin{scope} [xscale=0.6,xshift=12.3cm,yshift=0.15cm]
%alpha
\draw[red] (2.8,-0.2) ..controls +(0.2,-0.5) and +(0.2,0.5) .. (2.8,-1.5);
\draw[dashed,red] (2.8,-0.2) ..controls +(-0.2,-0.5) and +(-0.2,0.5) .. (2.8,-1.5);
\draw [->,red] (2.95,-0.8)--(2.95,-0.9);
\draw[red] (2.8,-1.5) node[below left] {$\scriptstyle{\alpha_3}$};
%beta
\draw[blue] (2.7,-0.1)ellipse(2 and 0.9);
\draw[blue,->] (4.7,-0.1)--(4.7,-0.05) node[right] {$\scriptstyle{\beta_3}$};
%gamma
\draw[green,->] (2.3,-0.2) .. controls +(0.1,-0.1) and +(-0.2,0) .. (2.6,-0.5) .. controls +(0.5,0) and +(0,-0.5) .. (4.3,0) .. controls +(0,0.4) and +(0.5,0) .. (2.7,0.5) node[below] {$\scriptstyle{\gamma_3}$};
\draw[green] (2.7,0.5) .. controls +(-0.5,0) and +(0,0.4) .. (1.1,0) .. controls +(0,-0.5) and +(0,0.5) .. (2,-1.42);
\draw[green,dashed] (2,-1.42) .. controls +(0.2,0.5) and +(-0.2,-0.5) .. (2.3,-0.2);
\end{scope}
\end{tikzpicture} \caption{A trisection diagram for $(S^1\times S^2)\sharp(S^1\times S^2)\sharp \C P^2$} \label{figtrisectionexample}
\end{center}
\end{figure}
The trisection diagram $(\Sigma;\alpha,\beta,\gamma)$ in Figure \ref{figtrisectionexample} represents the $4$--manifold $X = (S^1\times S^3) \, \sharp \, (S^1\times S^3) \, \sharp \, \C P^2$. 
The black paths fix a choice of a representative in $\pi_1(\Sigma,\star)$ of each loop. 
Let $x_i, y_i$ for $i\in\{1,2,3\}$ be the generators of $\pi_1(\Sigma,\star)$ represented in Figure \ref{figbasisSigma}. Their homology classes provide a symplectic basis of $H_1(\Sigma;\Z)$. Note that the family $(x_i,y_i)_{1\leq i\leq 3}$ is not a symplectic basis for $\Sigma$ as in Definition \ref{defsympl}, although it could easily be modified to get such a basis.
The following relations hold in $\pi_1(\Sigma,\star)$: $\alpha_1=\beta_1=\gamma_1=x_1$, $\alpha_2=\beta_2=\gamma_2=x_2$, $\alpha_3=x_3$, $\beta_3=y_3$ and $\gamma_3=x_3y_3$.
\begin{figure}[htb]
\begin{center}
\begin{tikzpicture} [scale=0.5]
\draw (0,0) ..controls +(0,1) and +(-2,1) .. (4,2);
\draw (4,2) ..controls +(2,-1) and +(-2,-1) .. (8,2);
\draw (8,2) ..controls +(2,1) and +(-1.2,0) .. (12,1.3);
\draw (0,0) ..controls +(0,-1) and +(-2,-1) .. (4,-2);
\draw (4,-2) ..controls +(2,1) and +(-2,1) .. (8,-2);
\draw (8,-2) ..controls +(2,-1) and +(-1.2,0) .. (12,-1.3);
\draw (14,2) ..controls +(2,1) and +(0,1) .. (18,0);
\draw (14,-2) ..controls +(2,-1) and +(0,-1) .. (18,0);
\draw (12,1.3) ..controls +(0.5,0) and +(-1,-0.5) .. (14,2);
\draw (12,-1.3) ..controls +(0.5,0) and +(-1,0.5) .. (14,-2);

\draw (9.12,-0.8) node {$\star$};

\draw (2,0) ..controls +(0.5,-0.25) and +(-0.5,-0.25) .. (4,0);
\draw (2.3,-0.1) ..controls +(0.6,0.2) and +(-0.6,0.2) .. (3.7,-0.1);
\draw (3,-0.2) ..controls +(2,-0.7) and +(-2,0) .. (9.12,-0.8);
\draw (3,-2.3) ..controls +(2,1.7) and +(-2,-0.4) .. (9.12,-0.8);
\draw[dashed] (3,-0.2) .. controls +(-0.2,-0.5) and +(-0.2,0.5) .. (3,-2.3);
\draw (9.12,-0.8) .. controls +(-1.5,0) and +(1,0) .. (3,0.8) .. controls +(-0.5,0) and +(0,0.5) .. (1.5,0) .. controls +(0,-0.5) and +(-1,0) .. (3,-0.8) .. controls +(1,0) and +(-1,-0.15) .. (9.12,-0.8);
\draw [->] (3.5,-1.92)--(3.45,-1.96) node[above] {$\scriptstyle{x_1}$};
\draw [->] (3,0.8)--(2.9,0.8) node[above] {$\scriptstyle{y_1}$};

\draw (8,0) ..controls +(0.5,-0.25) and +(-0.5,-0.25) .. (10,0);
\draw (8.3,-0.1) ..controls +(0.6,0.2) and +(-0.6,0.2) .. (9.7,-0.1);
\draw (9,-0.2) ..controls +(0.2,-0.5) and +(0.2,0.5) .. (9,-2.3);
\draw[dashed] (9,-0.2) ..controls +(-0.2,-0.5) and +(-0.2,0.5) .. (9,-2.3);
\draw (9.12,-0.8) .. controls +(1,0.2) and +(0,-0.5) .. (10.5,0) .. controls +(0,0.5) and +(0.5,0) .. (9,0.8) .. controls +(-0.5,0) and +(0,0.5) .. (7.5,0) .. controls +(0,-0.5) and +(-1,0.2) .. (9.12,-0.8);
\draw [->] (9.12,-1.8)--(9.1,-1.9) node[right] {$\scriptstyle{x_2}$};
\draw [->] (10.5,0)--(10.5,0.1) node[right] {$\scriptstyle{y_2}$};

\draw (14,0) ..controls +(0.5,-0.25) and +(-0.5,-0.25) .. (16,0);
\draw (14.3,-0.1) ..controls +(0.6,0.2) and +(-0.6,0.2) .. (15.7,-0.1);
\draw[dashed] (15,-0.2) ..controls +(0.2,-0.5) and +(0.2,0.5) .. (15,-2.3);
\draw (15,-0.2) ..controls +(-2,-0.7) and +(2,0) .. (9.12,-0.8);
\draw (15,-2.3) ..controls +(-2,1.7) and +(2,-0.4) .. (9.12,-0.8);
\draw (9.12,-0.8) .. controls +(1.5,0) and +(-1,0) .. (15,0.8) .. controls +(0.5,0) and +(0,0.5) .. (16.5,0) .. controls +(0,-0.5) and +(1,0) .. (15,-0.8) .. controls +(-1,0) and +(1,-0.15) .. (9.12,-0.8);
\draw [->] (14.5,-1.92)--(14.55,-1.96) node[above] {$\scriptstyle{x_3}$};
\draw [->] (15,0.8)--(14.9,0.8) node[above] {$\scriptstyle{y_3}$};
\end{tikzpicture}\caption{A basis of $H_1(\Sigma;\Z)$ with curves based at $\star$} \label{figbasisSigma}
\end{center}
\end{figure}

Setting $L=\langle x_1,x_2\rangle\subset H_1(\Sigma;\Z)$, we get $$L_\alpha=L\oplus\langle x_3\rangle, \ L_\beta=L\oplus\langle y_3\rangle \text{ and } L_\gamma=L\oplus\langle x_3+y_3\rangle.$$ Hence, by Theorem \ref{thhomologyXnotwist}:
\begin{itemize}
\item $H_1(X;\Z)\simeq\Z^2$ is generated by $y_1$ and $y_2$ 
\item $H_2(X;\Z)\simeq\Z$ is generated by $x_3$ 
\item  $H_3(X;\Z)\simeq\Z^2$ is generated by $x_1$ and $x_2$. 
\end{itemize}
In these bases, the matrix of the intersection form on $H_2(X;\Z)$ is $(1)$ and the matrix of the form on $H_1(X;\Z)\times H_3(X;\Z)$ is  $\begin{pmatrix} -1&0\\0&-1 \end{pmatrix}$, see Proposition \ref{propintform} and Proposition \ref{propintfromH13}.

Now let $G\simeq\Z^2$ be the free abelian (multiplicative) group of rank $2$ generated by $t_1$ and $t_2$. Let $\varphi :H_1(X;\Z)\to G$ be defined by $\varphi(y_1)=t_1$ and $\varphi(y_2)=t_2$. 
The following relations hold in $H_1^\varphi(\Sigma;R)$, assuming the lifts of the curves all start at the same lift of the point $\star$: $\alpha_1=\beta_1=\gamma_1=x_1$, $\alpha_2=\beta_2=\gamma_2=x_2$, $\alpha_3=x_3$, $\beta_3=y_3$ and $\gamma_3=x_3+y_3$.

In the cellular decomposition of $\Sigma$ given by Figure \ref{figbasisSigma}, the only 2--cell has boundary $\partial\Sigma=[x_1,y_1]\,y_2^{-1}\,[y_3^{-1},x_3]\,x_2\,y_2\,x_2^{-1}$. This provides a single relation in $H_1^\varphi(\Sigma;R)$:
$$r=(1-t_1)x_1+(t_2^{-1}-1)x_2.$$ 
Setting $L^\pr=\langle x_1,x_2\rangle/\langle r\rangle\subset H_1^\varphi(\Sigma;R)$, we get $$H_1^\varphi(\Sigma;R)=L^\pr\oplus\langle (1-t_2)y_1+(t_1-1)y_2,x_3,y_3\rangle$$ and
$$L_\alpha^\pr=L^\pr\oplus\langle x_3\rangle, \ L_\beta^\pr=L^\pr\oplus\langle y_3\rangle \text{ and } L_\gamma^\pr=L^\pr\oplus\langle x_3+y_3\rangle.$$ Hence by Theorem \ref{thtwistedhomologyX}:
\begin{itemize}
\item $H_1^\varphi(X;R)\simeq R$ is generated by $(1-t_2)\,y_1+(t_1-1)\,y_2$,
\item $H_2^\varphi(X;R)\simeq R$ is generated by $x_3$,
\item $H_3^\varphi(X;R)\simeq L^\pr$ and $H_3^\varphi(X;\F)\simeq\F$ is generated by $x_1$.
\end{itemize}
In these generators, the intersection form on $H_2^\varphi(X;\F)\simeq \F$ is given by $1$ and the intersection form on $H_1^\varphi(X;\F)\times H_3^\varphi(X;\F)\simeq\F\times\F$ is given by $t_1(t_2-1)$, see Theorem \ref{thtwistedintform} and Proposition \ref{propintfromH13twisted}.

We end with the computation of the torsion. Fix the homology basis $h=(h_3,h_2,h_1)$ with $h_3= x_1$, $h_2=x_3$, $h_1=(1-t_2)\,y_1+(t_1-1)\,y_2$. We compute the torsion $\tau^\varphi(X;h)\in\F/\Lambda^*$.  
Set $u=y_1$. By Proposition \ref{propTorsXpX}, $\tau^\varphi(X;h)=(t_1-1)^{-1}\tau^\varphi(\p X,\star;\p h)$ where $\p h=(\p h_3,\p h_2,\p h_1)$ and $\p h_3=(r,h_3)$, $\p h_2=h_2$, $\p h_1=(h_1,u)$. 
By Theorem \ref{thtorsionXstar}, $\tau^\varphi(\p X,\star;\p h)$ equals the torsion $\tau(\p C^\pf;\p c,\p h)$ of the complex $\p C^\pf$:
$$0 \to \p L^\pf\oplus \p L^\pf \oplus \p L^\pf \fl{\zeta} \p L_\alpha^\pf \oplus \p L_\beta^\pf \oplus \p L_\gamma^\pf \fl{\iota} H_1^\varphi(\p\Sigma,\star;\F) \to 0,$$
where $\p c$ is a $\Lambda$--basis over $\F$ and $\p L^\pf=\langle x_1,x_2\rangle\subset H_1^\varphi(\p\Sigma,\star;\F)$. Define $\p c=(\p c_3,\p c_2,\p c_1)$ by 
\begin{align*}
 &\p c_3=\big((x_1,0,0),(x_2,0,0),(0,x_1,0),(0,x_2,0),(0,0,x_1),(0,0,x_2)\big), \\
 &\p c_2=\big((x_1,0,0),(x_2,0,0),(x_3,0,0),(0,x_1,0),(0,x_2,0),(0,y_3,0),(0,0,x_1),(0,0,x_2),\\
 &\hspace{12cm} (0,0,x_3+y_3)\big), \\
 &\p c_1=(x_1,x_2,x_3,y_1,y_2,y_3).
\end{align*}
Also fix the following bases of $\mathrm{Im}(\zeta)$ and $\mathrm{Im}(\iota)$:
\begin{align*}
 &\p b_2=\big((x_1,-x_1,0),(x_2,-x_2,0),(0,x_1,-x_1),(0,x_2,-x_2)\big), \\
 &\p b_1=(x_1,x_2,x_3,y_3).
\end{align*}
Lift the latter two bases to get the following independant families in $\p L^\pf\oplus \p L^\pf \oplus \p L^\pf$ and $\p L_\alpha^\pf \oplus \p L_\beta^\pf \oplus \p L_\gamma^\pf$:
\begin{align*}
 &\bar b_2=\big((x_1,0,0),(x_2,0,0),(0,x_1,0),(0,x_2,0)\big), \\
 &\bar b_1=\big((x_1,0,0),(x_2,0,0),(x_3,0,0),(0,y_3,0)\big).
\end{align*}
Now, by definition of the torsion:
$$\tau\left(\p C^\pf;\p c,\p h\right)=\left[\frac{\p h_3.\bar b_2}{\p c_3}\right]\cdot\left[\frac{\p b_2.\p h_2.\bar b_1}{\p c_2}\right]^{-1}\left[\frac{\p b_1.\p h_1}{\p c_1}\right].$$
A straightforward computation gives $\tau^\varphi(X;h)=1-t_2\in\F/\Lambda^*$.

\subsection*{Example 2}
\begin{figure}[htb]
\begin{center}
\begin{tikzpicture} 
 \newcommand{\trou}[2]{\draw[xshift=#1cm,yshift=#2cm] (-1,0.1) ..controls +(0.5,-0.25) and +(-0.5,-0.25) .. (1,0.1) (-0.7,0) ..controls +(0.6,0.2) and +(-0.6,0.2) .. (0.7,0)}
 \draw[rounded corners=50pt] (0,5) -- (0,0) -- (12,0) -- (12,10) -- (0,10) -- (0,5);
 \trou62; \trou3{7.5}; \trou9{7.5}; \trou65;
 %alpha1
 \draw[red] (6,1.9) ..controls +(0.2,-0.5) and +(0.2,0.5) .. (6,0) node[below] {$\scriptstyle{\alpha_1}$};
 \draw[dashed,red] (6,1.9) ..controls +(-0.2,-0.5) and +(-0.2,0.5) .. (6,0);
 \draw[->,red] (6.15,1.1) -- (6.15,1);
 %alpha2
 \draw[red] (5.3,5) ..controls +(-1.5,0.6) and +(1.5,0.6) .. (0,5) node[left] {$\scriptstyle{\alpha_2}$};
 \draw[dashed,red] (0,5) ..controls +(1.5,-0.6) and +(-1.5,-0.6) .. (5.3,5);
 \draw[->,red] (2.7,5.45) -- (2.6,5.45);
 %alpha3
 \draw[red] (8.3,7.5) ..controls +(-1.5,0.6) and +(1.5,0.6) .. (3.7,7.5);
 \draw[dashed,red] (3.7,7.5) ..controls +(1.5,-0.6) and +(-1.5,-0.6) .. (8.3,7.5);
 \draw[->,red] (5.9,7.95) -- (6,7.95) node[above] {$\scriptstyle{\alpha_3}$};
 %alpha4
 \draw[red,->] (10,6) arc (-90:90:1.5) -- (2,9) arc (90:270:1.5) -- (2,6) -- (10,6) node[below] {$\scriptstyle{\alpha_4}$};
 %beta1
 \draw[blue] (6,5.15) ..controls +(0.6,1.5) and +(0.6,-1.5) .. (6,10) node[above] {$\scriptstyle{\beta_1}$};
 \draw[dashed,blue] (6,10) ..controls +(-0.6,-1.5) and +(-0.6,1.5) .. (6,5.15);
 \draw[->,blue] (6.45,7.5) -- (6.45,7.6);
 %beta2
 \draw[blue] (6.3,2.1) ..controls +(1,1) and +(0,-1.5) .. (9,7.4);
 \draw[dashed,blue] (9,7.4) ..controls +(-1,-1) and +(0,1.5) .. (6.3,2.1);
 \draw[->,blue] (7.18,3.2) -- (7.25,3.3) node[right] {$\scriptstyle{\beta_2}$};
 %beta3
 \draw[blue] (5.7,2.1) ..controls +(-1,1) and +(0,-1.5) .. (3,7.4);
 \draw[dashed,blue] (3,7.4) ..controls +(1,-1) and +(0,1.5) .. (5.7,2.1);
 \draw[->,blue] (4.82,3.2) -- (4.75,3.3) node[left] {$\scriptstyle{\beta_3}$};
 %beta4
 \draw[blue] (6,0.7) ..controls +(3,0) and +(0,-4) .. (11.7,7.5) ..controls +(0,1.5) and +(1.5,0) .. (9,9.5) ..controls +(-1,0) and +(0,1.5) .. (6.8,7.5) ..controls +(0,-1) and +(0,1) .. (7.2,5.2) ..controls +(0,-0.3) and +(0.6,0) .. (6,4.5);
 \draw[blue] (6,0.7) ..controls +(-3,0) and +(0,-4) .. (0.3,7.5) ..controls +(0,1.5) and +(-1.5,0) .. (3,9.5) ..controls +(1,0) and +(0,1.5) .. (5.2,7.5) ..controls +(0,-1) and +(0,1) .. (4.8,5.2) ..controls +(0,-0.3) and +(-0.6,0) .. (6,4.5);
 \draw[->,blue] (8,1.13) -- (8.1,1.17) node[below] {$\scriptstyle{\beta_4}$};
 %gamma1
 \draw[green] (6,2) ellipse (2 and 1.1);
 \draw[->,green] (8,1.9) -- (8,2) node[right] {$\scriptstyle{\gamma_1}$};
 %gamma2
 \draw[green] (6.7,5) ..controls +(1.5,0.6) and +(-1.5,0.6) .. (12,5) node[right] {$\scriptstyle{\gamma_2}$};
 \draw[dashed,green] (12,5) ..controls +(-1.5,-0.6) and +(1.5,-0.6) .. (6.7,5);
 \draw[->,green] (9.3,5.45) -- (9.4,5.45);
 %gamma3
 \draw[green] (9,7.64) .. controls +(0.5,0.3) and +(0,0.4) .. (10.3,7.5) .. controls +(0,-0.3) and +(0.6,0) .. (9,7.1) .. controls +(-0.6,0) and +(0,-0.3) .. (7.8,7.5) .. controls +(0,0.3) and +(-0.6,0) .. (9,8.1) .. controls +(1,0)and +(0,0.4) .. (10.7,7.5) .. controls +(0,-0.4) and +(1,0) .. (9,6.7) .. controls +(-1,0) and +(0,-0.4) .. (7.4,7.5) .. controls +(0,0.4) and +(-1,0) .. (9,8.5) .. controls +(1.2,0)and +(0,0.6) .. (11.1,7.5) .. controls +(0,-0.6) and +(1.2,0) .. (9,6.3) .. controls +(-1.2,0) and +(0,-0.6) .. (7,7.5) .. controls +(0,0.8) and +(-0.7,-1) .. (8.7,10) node[above] {$\scriptstyle{\gamma_3}$};
 \draw[dashed,green] (8.7,10) ..controls +(0.2,-1) and +(0,1) .. (9,7.64);
 \draw[->,green] (8.48,9.7) -- (8.55,9.8);
 %gamma4
 \draw[green] (3,7.64) .. controls +(-0.5,0.3) and +(0,0.4) .. (1.7,7.5) .. controls +(0,-0.3) and +(-0.6,0) .. (3,7.1) .. controls +(0.6,0) and +(0,-0.3) .. (4.2,7.5) .. controls +(0,0.3) and +(0.6,0) .. (3,8.1) .. controls +(-1,0)and +(0,0.4) .. (1.3,7.5) .. controls +(0,-0.4) and +(-1,0) .. (3,6.7) .. controls +(1,0) and +(0,-0.4) .. (4.6,7.5) .. controls +(0,0.4) and +(1,0) .. (3,8.5) .. controls +(-1.2,0)and +(0,0.6) .. (0.9,7.5) .. controls +(0,-0.6) and +(-1.2,0) .. (3,6.3) .. controls +(1.2,0) and +(0,-0.6) .. (5,7.5) .. controls +(0,0.8) and +(0.7,-1) .. (3.3,10) node[above] {$\scriptstyle{\gamma_4}$};
 \draw[dashed,green] (3.3,10) ..controls +(-0.2,-1) and +(0,1) .. (3,7.64);
 \draw[->,green] (3.52,9.7) -- (3.45,9.8);
 %PtBase
 \draw (6,4) node {$\star$};
 \draw[red!30] (6,4) .. controls +(-3,0) and +(-4,0) .. (6.15,0.8);
 \draw[red!30] (6,4) .. controls +(-0.6,0) and +(0.5,-0.7) .. (4.8,5) .. controls +(-0.5,0.7) and +(0,-1) .. (5.1,7.5) .. controls +(0,0.2) and +(0.05,-0.1) .. (5.05,7.87);
 \draw[blue!30] (6,4) .. controls +(-2.5,0) and +(1,1) .. (2.3,2.3);
 \draw[blue!30] (6,4) -- (7.5,3.7);
 \draw[blue!30] (6,4) .. controls +(0.6,0) and +(-0.5,-0.7) .. (7.2,5) .. controls +(0.4,0.6) and +(0.5,0) .. (6.1,5.5);
 \draw[green!30] (6,4) -- (6,3.1);
 \draw[green!30] (6,4) .. controls +(1.5,0) and +(-1,-2) .. (8.5,6.35);
 \draw[green!30] (6,4) .. controls +(-1.5,0) and +(1,-2) .. (3.5,6.35);
\end{tikzpicture} \caption{A trisection diagram for $S^1\times L(3,1)$} \label{figtrisectionex}
\end{center}
\end{figure}
The trisection diagram $(\Sigma;\alpha,\beta,\gamma)$ in Figure \ref{figtrisectionex} represents the $4$--manifold $X = S^1\times L(3,1)$ product of a circle with the Lens space $L(3,1)$, see \cite[Figure~10]{Koenig}. 
Generators $x_i, y_i$ of $\pi_1(\Sigma,\star)$, with $i\in\{1,\dots,4\}$, are given in Figure \ref{figtrisectionexbase}.
\begin{figure}[htb]
\begin{center}
\begin{tikzpicture} 
 \newcommand{\trou}[2]{\draw[xshift=#1cm,yshift=#2cm] (-1,0.1) ..controls +(0.5,-0.25) and +(-0.5,-0.25) .. (1,0.1) (-0.7,0) ..controls +(0.6,0.2) and +(-0.6,0.2) .. (0.7,0)}
 \draw[rounded corners=50pt] (0,5) -- (0,0) -- (12,0) -- (12,10) -- (0,10) -- (0,5);
 \trou62; \trou3{7.5}; \trou9{7.5}; \trou65;
 %PtBase
 \draw (6,4) node {$\star$};
 %x1
 \draw (6,4) .. controls +(-2,-1) and +(-1,2) .. (5,0) (5.7,1.92) .. controls +(-0.5,-0.5) and +(0,-0.5) .. (4.9,2) .. controls +(0,0.5) and +(-1,-1) .. (6,4);
 \draw[dashed] (5,0) .. controls +(0.7,1) and +(0,-0.7) .. (5.7,1.92);
 \draw[->] (4.345,2) -- (4.345,1.9) node[left] {$x_1$};
 %x2
 \draw (6,5.15) .. controls +(1,0.3) and +(2,1.6) .. (6,4) .. controls +(3,1.5) and +(1.5,-4) .. (6,10);
 \draw[dashed] (6,10) .. controls +(-0.7,-2) and +(-0.7,2) .. (6,5.15);
 \draw[->] (6.4,9) -- (6.36,9.1) node[right] {$x_2$};
 %x3
 \draw[->] (9,4.5) .. controls +(0.8,0.2) and +(-1,-1) .. (12,6) (9,7.4) .. controls +(0,-1) and +(3,1) .. (6,4) .. controls +(1.5,0.3) and +(-0.8,-0.2) .. (9,4.5) node[below] {$x_3$};
 \draw[dashed] (12,6) .. controls +(-1,0) and +(1,-1) .. (9,7.4);
 %x4
 \draw[->] (4.8,7.5).. controls +(0,1) and +(0.5,-1) .. (4,10) (3.5,7.56) .. controls +(0.8,0.3) and +(0,0.3) .. (4.3,7.5) .. controls +(0,-0.7) and +(-4,1.3) .. (6,4) .. controls +(-1,0.4) and +(0,-0.5) .. (4.3,5.5) .. controls +(0,1) and +(0,-1) .. (4.8,7.5) node[right] {$x_4$};
 \draw[dashed] (4,10) .. controls +(-0.5,-1) and +(-0.3,1) .. (3.5,7.56);
 %y1
 \draw[->] (6,1) .. controls +(2,0) and +(1.5,-1) .. (6,4) .. controls +(-1.5,-1) and +(-2,0) .. (6,1) node[below] {$y_1$};
 %y2
 \draw[->] (6,6) .. controls +(-2,0) and +(-1.7,1) .. (6,4) .. controls +(1.7,1) and +(2,0) .. (6,6) node[above] {$y_2$};
 %y3
 \draw[->] (9,8.5) .. controls +(-0.7,0) and +(0,0.7) .. (7.5,7.5) .. controls +(0,-1) and +(0,1) .. (8,5.5) .. controls +(0,-0.5) and +(1,0.3) .. (6,4) .. controls +(3,1) and +(0,-2) .. (10.5,7.5) .. controls +(0,0.7) and +(0.7,0) .. (9,8.5) node[above] {$y_3$};
 %y4
 \draw[<-] (3,8.5) .. controls +(0.7,0) and +(0,0.7) .. (4.5,7.5) .. controls +(0,-1) and +(0,1) .. (4,5.5) .. controls +(0,-0.5) and +(-1,0.3) .. (6,4) .. controls +(-3,1) and +(0,-2) .. (1.5,7.5) .. controls +(0,0.7) and +(-0.7,0) .. (3,8.5) node[above] {$y_4$};
\end{tikzpicture} \caption{A basis for $\pi_1(\Sigma,\star)$} \label{figtrisectionexbase}
\end{center}
\end{figure}
Their homology classes provide a symplectic basis of $H_1(\Sigma;\Z)$. 
The following relations hold in $\pi_1(\Sigma,\star)$:
$$
\begin{array}{lclcl}
 \alpha_1=x_1 & \quad & \beta_1=x_2 & \quad & \gamma_1=y_1 \\
 \alpha_2=y_4^{-1}\,x_4\,y_4\,x_4^{-1}\,y_2^{-1}\,x_2\,y_2 && \beta_2=x_3^{-1}\,y_1^{-1}\,x_1\,y_1 && \gamma_2=x_3\,y_3^{-1}\,x_3^{-1}\,y_3\,x_2 
 \\
 \alpha_3=y_2^{-1}\,y_3^{-1}\,x_3^{-1}\,y_3\,y_2\,x_4 && \beta_3=y_4^{-1}\,x_4^{-1}\,y_4\,x_1 && \gamma_3=y_3^{-1}\,x_3\,y_3^{-2} \\
 \alpha_4=y_2^{-1}\,y_3\,y_2\,y_4 && \beta_4=y_1\,y_3\,y_4 && \gamma_4=x_4\,y_4^3
\end{array}$$
We obtain in $H_1(\Sigma;\Z)$:
$$
\begin{array}{lclcl}
 \alpha_1=x_1 & \ \  & \beta_1=x_2 & \ \  & \gamma_1=y_1 \\
 \alpha_2=x_2 && \beta_2=x_1-x_3 && \gamma_2=x_2 \\
 \alpha_3=-x_3+x_4 && \beta_3=x_1-x_4 && \gamma_3=x_3-3y_3 \\
 \alpha_4=y_3+y_4 && \beta_4=y_1+y_3+y_4 && \gamma_4=x_4+3y_4 
\end{array}$$

Hence, by Theorem \ref{thhomologyXnotwist}
\begin{itemize}
\item $H_1(X;\Z)\simeq\Z\oplus\frac{\Z}{3\Z}$ with the first summand generated by $y_2$ and the second by~$y_3$, 
\item $H_2(X;\Z)\simeq\frac{\Z}{3\Z}$ is generated by $y_3+y_4$, 
\item  $H_3(X;\Z)\simeq\Z$ is generated by $x_2$. 
\end{itemize}

In these bases, the intersection form on $H_1(X;\Z)\times H_3(X;\Z)$ is  given by $\langle y_2,x_2\rangle=-1$.

Now let $G\simeq\Z$ be the multiplicative group generated by $t$. Let $\varphi :H_1(X;\Z)\to G$ be defined by $\varphi(y_2)=t$ and $\varphi(y_3)=1$. 
The following relations hold in $H_1^\varphi(\Sigma;R)$, with $R=\Zt\textrm{ or }\Q(t)$, assuming the lifts of the curves all start at the same lift of the point~$\star$:
$$
\begin{array}{lclcl}
 \alpha_1=x_1 & \ \  & \beta_1=x_2 & \ \  & \gamma_1=y_1 \\
 \alpha_2=t^{-1}x_2 && \beta_2=x_1-x_3 && \gamma_2=x_2 \\
 \alpha_3=-t^{-1}x_3+x_4 && \beta_3=x_1-x_4 && \gamma_3=x_3-3y_3 \\
 \alpha_4=t^{-1}y_3+y_4 && \beta_4=y_1+y_3+y_4 && \gamma_4=x_4+3y_4 
\end{array}$$

In the cellular decomposition of $\Sigma$ given by Figure \ref{figtrisectionexbase}, the only 2--cell has boundary $\partial\Sigma=[y_1^{-1},x_1]\,[y_4^{-1},x_4]\,[y_2^{-1},x_2]\,[y_3^{-1},x_3]$. This provides a single relation in $H_1^\varphi(\Sigma;R)$:
$r=(1-t)\,x_2$. Hence by Theorem \ref{thtwistedhomologyX}:
\begin{itemize}
 \item $H_1^\varphi(X;\Zt)\simeq \frac{\Z}{3\Z}$ is generated by $y_3$ and $H_1^\varphi(X;\Q(t))=0$,
 \item $H_2^\varphi(X;R)=0$,
 \item $H_3^\varphi(X;\Zt)\simeq \frac{\Zt}{\langle 1-t\rangle}$ is generated by $x_2$ and $H_3^\varphi(X;\Q(t))=0$.
\end{itemize}

We now compute the torsion. Set $\Lambda=\Zt$ and $\F=\Q(t)$. The manifold $X$ has no homology over~$\F$. Set $u=y_2$. By Proposition~\ref{propTorsXpX}, the torsion $\tau^\varphi(X)\in\F/\Lambda^*$ is given by $\tau^\varphi(X)=(t-1)^{-1}\tau^\varphi(\p X,\star;\p h)$ where $\p h=(r,\emptyset,u)$. 
By Theorem~\ref{thtorsionXstar}, $\tau^\varphi(\p X,\star;\p h)$ equals the torsion $\tau(\p C^\pf;\p c,\p h)$ of the complex $\p C^\pf$:
$$0 \to \oplus_{\nu\neq\nu'}\left( \p L_\nu^\pf\cap \p L_{\nu'}^\pf \right) \fl{\zeta} \oplus_{\nu} \p L_\nu^\pf \fl{\iota} H_1^\varphi(\p\Sigma,\star;\F) \to 0,$$
where $\nu,\nu'$ run over $\{\alpha,\beta,\gamma\}$ and $\p c$ is a $\Lambda$--basis over $\F$. 
$$
\begin{array}{ll}
 \p L_\alpha^\F=\langle x_2,x_1,x_3-tx_4,y_3+ty_4\rangle & \p L_\alpha^\F\cap\p L_\beta^\F=\langle x_2,(t-1)x_1+x_3-tx_4\rangle \\
 \p L_\beta^\F=\langle x_2,x_1-x_3,x_1-x_4,y_1+y_3+y_4\rangle & \p L_\beta^\F\cap\p L_\gamma^\F=\langle x_2,-x_3+x_4+3(y_1+y_3+y_4)\rangle \\
 \p L_\gamma^\F=\langle x_2,y_1,x_3-3y_3,x_4+3y_4\rangle & \p L_\gamma^\F\cap\p L_\alpha^\F=\langle x_2,x_3-tx_4-3y_3-3ty_4\rangle
\end{array}
$$

As in the first example, fix bases $\p c$, $\p b$ and $\bar b$ to compute $\tau(\p C^\pf;\p c,\p h)$. The computation gives $\tau^\varphi(X)=1-t\in\F/\Lambda^*$.

\providecommand{\MR}[1]{}
\providecommand{\bysame}{\leavevmode\hbox to3em{\hrulefill}\thinspace}
\providecommand{\MR}{\relax\ifhmode\unskip\space\fi MR }
\providecommand{\MRhref}[2]{%
  \href{http://www.ams.org/mathscinet-getitem?mr=#1}{#2}
}
\providecommand{\href}[2]{#2}


\begin{thebibliography}{FKSZ18}

\bibitem[Bla57]{Bl57}
R.~C. Blanchfield, \emph{Intersection theory of manifolds with operators with
  applications to knot theory}, Annals of Mathematics. Second Series
  \textbf{65} (1957), 340--356.

\bibitem[FKSZ18]{FKSZ18}
P.~Feller, M.~Klug, T.~Schirmer, and D.~Zemke, \emph{Calculating the homology
  and intersection form of a $4$--manifold from a trisection diagram}, arXiv:
  1711.04762 (2018).

\bibitem[GK16]{GK}
D.~Gay and R.~Kirby, \emph{Trisecting 4--manifolds}, Geometry \& Topology
  \textbf{20} (2016), no.~6, 3097--3132.

\bibitem[KL99]{KL}
Paul Kirk and Charles Livingston, \emph{Twisted {A}lexander invariants,
  {R}eidemeister torsion, and {C}asson-{G}ordon invariants}, Topology
  \textbf{38} (1999), no.~3, 635--661.

\bibitem[Koe17]{Koenig}
D.~Koenig, \emph{Trisections of 3-manifold bundles over $S^1$}, arXiv: 1710.04345 (2017).

\bibitem[LP72]{LP}
F.~Laudenbach and V.~Po{\'e}naru, \emph{A note on 4-dimensional handlebodies},
  Bulletin de la Soci{\'e}t{\'e} Math{\'e}matique de France \textbf{100}
  (1972), 337--344.

\bibitem[Mil66]{Mi66}
J.~Milnor, \emph{Whitehead torsion}, Bulletin of the American Mathematical
  Society \textbf{72} (1966), 358--426.

\bibitem[Ran97]{Ra09}
A.~Ranicki, \emph{The maslov index and the {W}all signature non-additivity
  invariant}, Unpublished (1997),
  \href{https://www.maths.ed.ac.uk/~v1ranick/papers/maslov.pdf}{https://www.maths.ed.ac.uk/{$\scriptstyle{\sim}$}v1ranick/papers/maslov.pdf}.

\bibitem[Rei39]{Re39}
K.~Reidemeister, \emph{Durchschnitt und {S}chnitt von {H}omotopieketten},
  Monatshefte f{\"u}r Mathematik und Physik \textbf{48} (1939), 226--239.

\bibitem[Tur01]{Tu01}
V.~Turaev, \emph{Introduction to combinatorial torsions}, Lectures in
  Mathematics ETH Z\"urich, Birkh\"auser Verlag, Basel, 2001, Notes taken by
  Felix Schlenk.

\bibitem[Wal69]{Wa69}
C.~T.~C. Wall, \emph{Non-additivity of the signature}, Inventiones Mathematicae
  \textbf{7} (1969), 269--274.

\end{thebibliography}
\end{document}